\documentclass[UTF-8,reqno]{amsart}
\usepackage{enumerate}
\usepackage{mhequ}
\usepackage{amssymb,url,color, booktabs,nccmath}
\usepackage[left=3.2cm,right=3.2cm,top=4cm,bottom=4cm]{geometry}
\usepackage{mathrsfs}
\usepackage{enumitem,dsfont}

\usepackage{color}
\usepackage[colorlinks=true]{hyperref}
\hypersetup{
	linkcolor=blue,          
	citecolor=red,        
	filecolor=blue,      
	urlcolor=cyan
}

\definecolor{darkergreen}{rgb}{0.0, 0.5, 0.0}

\numberwithin{equation}{section}

\newcommand{\be}{\begin{eqnarray}}
	\newcommand{\ee}{\end{eqnarray}}
\newcommand{\ce}{\begin{eqnarray*}}
	\newcommand{\de}{\end{eqnarray*}}
\newtheorem{theorem}{Theorem}[section]
\newtheorem{lemma}[theorem]{Lemma}
\newtheorem{proposition}[theorem]{Proposition}
\newtheorem{Examples}[theorem]{Example}
\newtheorem{corollary}[theorem]{Corollary}

\newtheorem{definition}[theorem]{Definition}
\theoremstyle{definition}
\newtheorem{remark}[theorem]{Remark}

\newcommand{\rmk}[1]{\textcolor{black}{#1}}

\newcommand{\assign}{:=}
\newcommand{\cdummy}{\cdot}

\DeclareMathOperator{\supp}{supp}

\def\eps{\varepsilon}

\def\u{\mathbf{u}}

\def\p{\partial}

\def\<{{\langle}}
\def\>{{\rangle}}
\def\({{\Big(}}
\def\){{\Big)}}

\def\bx{{\mathbf{x}}}

\def\dif{{\mathord{{\rm d}}}}

\def\min{{\mathord{{\rm min}}}}

\def\no{\nonumber}
\def\={&\!\!=\!\!&}

\def\cB{{\mathcal B}}

\def\cF{{\mathcal F}}
\def\cG{{\mathcal G}}

\def\mN{{\mathbb N}}

\def\mQ{{\mathbb Q}}
\def\mR{{\mathbb R}}

\def\mU{{\mathbb U}}

\def\bP{{\mathbf P}}

\def\1{{\mathbf{1}}}

\def\E{\mathbf E}

\def\geq{\geqslant}
\def\leq{\leqslant}

\def\div{\mathord{{\rm div}}}

\def\eps{\varepsilon}

\def\u{\mathbf{u}}

\def\p{\partial}

\def\<{{\langle}}
\def\>{{\rangle}}
\def\({{\Big(}}
\def\){{\Big)}}

\def\bx{{\mathbf{x}}}

\def\dif{{\mathord{{\rm d}}}}

\def\min{{\mathord{{\rm min}}}}

\def\no{\nonumber}
\def\={&\!\!=\!\!&}
\def\bt{\begin{theorem}}
	\def\et{\end{theorem}}
\def\bl{\begin{lemma}}
	\def\el{\end{lemma}}
\def\br{\begin{remark}}
	\def\er{\end{remark}}
\def\bx{\begin{Examples}}
	\def\ex{\end{Examples}}
\def\bd{\begin{definition}}
	\def\ed{\end{definition}}
\def\bp{\begin{proposition}}
	\def\ep{\end{proposition}}
\def\bc{\begin{corollary}}
	\def\ec{\end{corollary}}

\def\geq{\geqslant}
\def\leq{\leqslant}

\def\div{\mathord{{\rm div}}}

\def\bP{{\mathbf P}}

 \def\R{\mathbb R}
 \def\R{\mathbb R}    
\def\N{\mathbb N}  
   
\def\<{\langle} \def\>{\rangle}

\allowdisplaybreaks

\begin{document}

\title[Non-uniqueness of Leray--Hopf solutions]{Non-uniqueness of Leray--Hopf solutions for stochastic forced Navier--Stokes equations}
	\author{Martina Hofmanov\'a}
\address[M. Hofmanov\'a]{Fakult\"at f\"ur Mathematik, Universit\"at Bielefeld, D-33501 Bielefeld, Germany}
\email{hofmanova@math.uni-bielefeld.de}

\author{Rongchan Zhu}
\address[R. Zhu]{Department of Mathematics, Beijing Institute of Technology, Beijing 100081, China} 
\email{zhurongchan@126.com}

\author{Xiangchan Zhu}
\address[X. Zhu]{ Academy of Mathematics and Systems Science,
	Chinese Academy of Sciences, Beijing 100190, China}
\email{zhuxiangchan@126.com}
\thanks{
	M.H. is grateful for funding from the European Research Council (ERC) under the European Union's Horizon 2020 research and innovation programme (grant agreement No. 949981) and 	the financial supports  by the Deutsche Forschungsgemeinschaft (DFG, German Research Foundation) – Project-ID 317210226--SFB 1283. R.Z. and X.Z. are grateful
	to the financial supports by National Key R\&D Program of China (No. 2022YFA1006300).
	R.Z. is grateful to the financial supports of the NSFC (No.  12271030) and BIT Science and Technology Innovation Program Project 2022CX01001.
	X.Z. is grateful to the financial supports  in part by National Key R\&D Program of China (No. 2020YFA0712700) and the NSFC (No.  12090014, 12288201) and
	the support by key Lab of Random Complex Structures and Data Science,
	Youth Innovation Promotion Association (2020003), Chinese Academy of Science. Rongchan Zhu is the corresponding author.
}

\begin{abstract}
We study the question of non-uniqueness of Leray--Hopf solutions to  stochastic forced Navier--Stokes equations on $\R^{3}$ starting from zero initial condition. Specifically, we consider a linear multiplicative noise and the equations are perturbed by an additional body force $f$. This type of noise  is particularly appealing due to the regularization by noise phenomena established in \cite{RZZ14}, which provides global uniqueness for arbitrary $f\in L^{2}(0,T;H^{s-1})$ for $s>3/2$ with high probability. Based on the ideas of Albritton, Bru\'e and Colombo \cite{ABC22}, we prove that there exists a force $f\in L^{1}(0,T;L^{p})$, $p<3$, so that  non-uniqueness of local-in-time probabilistically strong Leray--Hopf solutions as well as joint non-uniqueness in law of Leray--Hopf solutions on $\R^{+}$ hold true. In the deterministic setting, we show that the set of forces, for which Leray--Hopf solutions are non-unique, is dense in $L^{1}(0,T;L^{2})$. In addition, by a simple controllability argument we show that for every divergence-free initial condition in $L^{2}$ there is a force so that non-uniqueness of Leray--Hopf solutions holds.
 \end{abstract}

\date{\today}

\maketitle

\tableofcontents

\section{Introduction}

We address the issue of non-uniqueness of Leray--Hopf solutions to stochastic forced Navier--Stokes equations on $\R^{3}$. The equations take the form
\begin{equation}
	\label{1:intro}
	\aligned
	\dif u+\div(u\otimes u)\dif t+\nabla p\dif t&=\Delta u \dif t+ u\dif W+f\dif t,
	\\\div u&=0,
	\endaligned
\end{equation}
where $W$ is a real-valued Wiener process on some stochastic basis $(\Omega,\cF, (\cF_t)_{t\geq 0}, \bP)$.  Our primary focus is on the linear multiplicative noise which has attracted a lot of attention in the literature. Of particular relevance to our investigation  is the regularization by noise phenomena obtained in \cite{RZZ14} (also explored in  \cite{GHV14} for the case of Euler equations), with which  we  later reconcile our findings.

 In essence, Leray--Hopf  solutions adhere to the Navier--Stoke equations \eqref{1:intro} in the analytically weak sense, while also satisfying the crucial requirement of the energy inequality, rendering them highly pertinent from a physical standpoint. Formally, the corresponding energy equality is derived by  testing the Navier--Stokes equations with the solution itself or, in other words, by applying It\^o's formula to the $L^{2}$-norm of the solution. However, due to the limited regularity of weak solutions, executing this testing procedure rigorously proves challenging; hence, a preliminary approximation method  is necessary, followed by the application of lower semicontinuity, yielding only an inequality in the limit. Various formulations of the energy inequality exist, and we defer the precise definition to Section~\ref{sec:3.2}.

Even in the absence of stochastic forcing, the (non)uniqueness of Leray--Hopf solutions has historically represented a prominent open problem in fluid dynamics,  that was only recently resolved in the groundbreaking work  \cite{ABC22}. Specifically, the authors demonstrated that for the initial condition $u_{0}=0$, there exists a time $T>0$ and a force $f\in L^{1}(0,T;L^{2})$ with two distinct Leray--Hopf solutions to the deterministic forced Navier--Stokes equations on $[0,T]\times\R^{3}$. The proof relies on the construction of a smooth, compactly supported steady state of the Navier--Stokes equations in similarity variables, which is inherently linearly unstable in dynamics, evidenced by the presence of an unstable eigenvalue in the corresponding linearized operator. The force $f$ is implicitly defined to ensure satisfaction of the Navier--Stokes equations. Consequently, the second solution follows a trajectory on an unstable manifold, with the two solutions differing in their decay as $t\to0^{+}$.

Prior to the seminal work  \cite{ABC22}, the non-uniqueness of weak solutions lacking the energy inequality to the deterministic forced Navier--Stokes equations was established in \cite{BV19a} through convex integration. However, the incorporation of the energy inequality was only achieved for the $p$-Navier--Stokes equations, where the Laplacian is replaced by the $p$-Laplacian with $p\in (1,6/5)$, as demonstrated in \cite{BMS20}, and for fractional Navier--Stokes equations, as seen in \cite{CDD18}. These findings served as the foundation for various non-uniqueness results concerning the stochastic counterpart of the equations, such as non-uniqueness in law \cite{HZZ19}, non-uniqueness of Markov solutions \cite{HZZ21markov}, and non-uniqueness of ergodic stationary solutions \cite{HZZ22}. Furthermore, also global existence and non-uniqueness when the system is driven by space-time white noise was established in \cite{HZZ21}. Stochastic power law fluids were addressed in \cite{LZ23}.

In this study, we show that the construction developed  in \cite{ABC22} can be suitably adapted to the framework of \eqref{1:intro}, thereby proving the following local-in-time result. For precise assumptions and  statement, we direct the reader to Theorem~\ref{th:2}.

\begin{theorem}\label{th:1intro}
There exists an $(\mathcal{F}_{t})$-stopping time $T>0$ and an $(\mathcal{F}_{t})$-adapted $f\in L^{1}(0,T;L^{2})$ $\bf{P}$-a.s. such that there are two distinct $(\mathcal{F}_{t})$-adapted Leray--Hopf solutions to the Navier--Stokes equations \eqref{1:intro} on $[0,T]\times\R^{3}$ with  zero initial datum.
\end{theorem}

Furthermore, we extend the aforementioned solutions to the entire time interval $\mathbb{R}^{+}$ utilizing the probabilistic extension technique introduced in \cite{HZZ19}. The method was originally devised to prolong the existence of convex integration solutions, which exist up to a stopping time, to global solutions. At the stopping time, the convex integration solutions were connected to Leray--Hopf solutions, which exist for  every divergence-free initial condition in $L^{2}$. In the deterministic context, such a connection is straightforward. However, in the stochastic setting, the subtlety arises from the probabilistic weakness of Leray--Hopf solutions; they cannot be constructed on every probability space with a given Wiener process, yet these probabilistic elements become integral components of the solution. The probabilistic extension method from \cite{HZZ19} has since been widely employed, particularly in scenarios where convex integration is applied within the stochastic framework (see, for instance, \cite{Ya20a, Ya20b}).

Unlike Leray--Hopf solutions, the convex integration solutions are probabilistically strong, meaning they are adapted to the given Wiener process. Remarkably, the solutions derived in Theorem~\ref{th:1intro} are also probabilistically strong. To facilitate the application of the probabilistic extension technique from \cite{HZZ19}, we initially elevate these solutions to probability measures on the canonical space of trajectories up to a stopping time. This canonical space is generated by both the solution and the noise. Subsequently, these probability measures are extended by the laws of classical Leray--Hopf solutions to encompass probability measures on trajectories over $\mathbb{R}^{+}$. The resulting solutions adhere to the energy inequality throughout $\mathbb{R}^{+}$ and are thus classified as Leray--Hopf solutions. Consequently, this allows us to establish joint non-uniqueness in law within this class, specifically, non-uniqueness of the joint laws of $(u,W)$. The proof of this result is provided in Theorem~\ref{th:1law}.

\begin{theorem}\label{th:2intro}
There exists $f$, a measurable functional of the driving Wiener process $W$, such that joint non-uniqueness in law holds true in the class of Leray--Hopf solutions.
\end{theorem}

As previously mentioned, we aim  to compare the above results with the available well-posedness theory. Specifically, we consider a form of noise that provides a regularization effect, as demonstrated in \cite{RZZ14}. The Navier--Stokes system under consideration in the latter work takes the following form
\begin{align}\label{eq:p}
\dif u+\div (u\otimes u)\dif t+\nabla p\dif t=\Delta u\dif t+\beta u\dif W,\quad \div u=0,
\end{align}
with $W$ a being Wiener process and $\beta\in \R$. It was established that local strong solutions exist and are unique for initial conditions in $H^{s}$ with $s>3/2$. Furthermore, for every $\varepsilon>0$ there exists $\kappa=\kappa(\beta^{2},\varepsilon)$ satisfying $\lim_{\beta\to\infty} \kappa(\beta^{2},\varepsilon)=\infty$ such that whenever $\|u_0\|_{H^s}\leq \kappa$ for some $s>3/2$ then the solution is global with probability bigger than $1-\varepsilon$.
The basic idea is that letting $v=e^{-\beta W}u$   It\^o's formula implies
\begin{align}\label{eq:po}
	\p_t v+\frac{\beta^2}2v+e^{\beta W}\div (v\otimes v)+\nabla p_v=\Delta v.
\end{align}
Here, the second term $\frac{\beta^2}2v$ furnishes sufficient dissipation to yield global solutions with high probability.

 To facilitate comparison with the findings of this paper, we re-examine the proof presented in \cite{RZZ14} and introduce an additional body force $f$ into \eqref{eq:p}. It emerges that the results established in \cite{RZZ14} remain applicable under the condition $f\in L^{2}(0,T;H^{s-1})$, or $f$ can also be random provided $f\in L^{2}(\Omega;L^{2}(0,T;H^{s-1}))$ and it is progressively measurable. Specifically, for any force in this function space, the solution originating from the zero initial condition is both unique and globally defined with a high probability.
However, this level of regularity does not extend to the force derived in Theorem~\ref{th:1intro} and Theorem~\ref{th:2intro} (as particularly evident in (1.24) of \cite{ABC22}), which belongs only to $L^{1}(0,T;L^{p})$ for every $p<3$. Consequently, such forces give rise to solutions that are distinct immediately after the initial time with full  probability.
Moreover, the integrability condition $L^{1}(0,T;L^{p})$ for every $p<3$ is sharp locally in time, mirroring the deterministic scenario: if $f\in L^{1}(0,T;L^{3})$, then uniqueness of solutions is assured for solutions within the class $C([0,T];L^{3})$, which exist locally in time (refer to Remark~\ref{r:7} for further elaboration).

Our proof of Theorem~\ref{th:1intro} and Theorem~\ref{th:2intro}  relies on a novel transformation of the Navier--Stokes system \eqref{1:intro} to a random partial differential equation. Unlike the conventional exponential transformation $v=e^{-W}u$ leading to an analogue of \eqref{eq:po}, we adopt a two-step transformation. This approach enables us to relocate the random and time-dependent coefficient away from the nonlinearity and towards the dissipative term, where it manifests as a viscosity.
To elaborate, we define $w(t):=e^{-(W(t)-t/2)}u(t)$ and $w(t)=:v(\int_{0}^{t}e^{W(s)-s/2}\mathrm{d}s)=:(v\circ\theta)(t)$. Notably, the function $\theta$ is invertible. By applying It\^o's formula, we obtain (refer to Section~\ref{sec:loc} for further details):
\begin{equation}\label{eq:v}
\partial_t v+\div (v\otimes v)+\nabla \pi=h(\theta^{-1}(t))\Delta v+g,
\end{equation}
for an appropriately defined pressure $\pi$, force $g$ and viscosity $h(\theta^{-1}(t))$.

In essence, we apply the methodology of \cite{ABC22} to the transformed equation \eqref{eq:v}. Notably, the modification solely affects the dissipative term, which is inherently treated as a perturbation. To address this perturbation, we establish a novel regularity estimate for the semigroup generated by the linearized operator (denoted as $e^{\tau L_{ss}}$ in Section~\ref{s:2.3}) utilizing a Littlewood--Paley decomposition and paraproducts. This approach enables us to construct two distinct Leray--Hopf solutions to \eqref{eq:v}. However, in order to obtain meaningful solutions upon retransformation to the original equation \eqref{1:intro}, it is imperative to ensure their adaptiveness with respect to the original filtration $(\mathcal{F}_{t})_{t\geq0}$ generated by the Wiener process $W$. This is the key difference to the deterministic scenario and it is rather delicate due to the random time rescaling $\theta$. In particular, it requires several  auxiliary results concerning stopping times, filtrations and adaptedness.

\br \label{additive}
Our methodology readily extends to the scenario of forced Navier--Stokes equations perturbed with additive noise. Specifically, employing the Da Prato--Debussche trick, wherein the equations are decomposed into a linear stochastic Stokes system and a nonlinear random PDE for the remainder (as demonstrated in, for instance, \cite{HZZ19}), analogous arguments to those presented in Lemma \ref{lem:per} below yield multiple solutions to the nonlinear equation originating from the zero initial condition. Moreover, the probabilistic extension of solutions from \cite{HZZ19} facilitates the extension of our findings outlined in Theorem~\ref{th:1intro} and Theorem~\ref{th:2intro} to this context. Notably, the proof in the case of additive noise is comparatively straightforward, as no alteration of filtration is required. For further details and an additional result in a hyperviscous setting, the interested reader is directed to \cite{BJLZ23}.
\er

\medskip

Finally, we present several results in the deterministic setting. \rmk{We consider the following equation:
\begin{equation}
	\label{1:intro1}
	\aligned
	\partial_t u+\div(u\otimes u)+\nabla p&=\Delta u +f,
	\\\div u&=0.
	\endaligned
\end{equation}}

\begin{theorem}\label{th:3intro}
In the deterministic setting, non-uniqueness of Leray--Hopf solutions holds in the following situations.
\begin{enumerate}
\item For every $g\in L^{1}(0,1;L^{2})$ and every $\varepsilon>0$ there exists $f\in L^{1}(0,1;L^{2})$ satisfying $$\|g-f\|_{L^{1}(0,1;L^{2})}\leq\varepsilon$$
such that Leray--Hopf solutions with zero initial condition to  \eqref{1:intro1} are non-unique.
\item For every divergence-free initial condition $u_{0}\in L^{2}$ there exists $f\in L^{1}(0,1;L^{2})$ such that Leray--Hopf solutions to  \eqref{1:intro1} with the initial condition $u_{0}$ are non-unique.
\end{enumerate}
\end{theorem}

The result in \emph{(1)} is achieved through Theorem~\ref{th:1} and Corollary~\ref{c:2} by a modification of the construction from \cite{ABC22}. The result in \emph{(2)} is proved in Theorem~\ref{thm:6} by a simple controllability argument in the spirit of \cite{Fla97} and no implicit modification of the proof of \cite{ABC22} is necessary.

The paper is organized as follows. In Section~\ref{sec:2} we introduce the notation, set the basic general assumptions on the operator $G$ and recall the key elements of the construction from \cite{ABC22}. The case of linear multiplicative noise is treated in Section~\ref{s:3}.
Section~\ref{sec:det} and Section~\ref{s:6} then focus on the deterministic setting.

\bigskip

\section{Preliminaries}\label{sec:2}

Throughout the paper, we use the notation $a\lesssim b$ if there exists a constant $c>0$ such that $a\leq cb$, and we write $a\simeq b$ if $a\lesssim b$ and $b\lesssim a$.

\subsection{Function spaces}
Given a Banach space $E$ with a norm $\|\cdot\|_E$ and $T>0$, we write $C_TE=C([0,T];E)$ for the space of continuous functions from $[0,T]$ to $E$, equipped with the supremum norm $\|f\|_{C_TE}=\sup_{t\in[0,T]}\|f(t)\|_{E}$. We also use  $C([0,\infty);E)$ to denote the space of continuous functions from $[0,\infty)$ to $E$.  For $p\in [1,\infty]$ we write $L^p_TE=L^p(0,T;E)$ for the space of $L^p$-integrable functions from $[0,T]$ to $E$, equipped with the usual $L^p$-norm. We also use $L^p_{\text{loc}}(\mathbb{R}^+;E)$ to denote the space of functions $f$ from $[0,\infty)$ to $E$ satisfying $f|_{[0,T]}\in L^p_TE$ for all $T>0$.  Similar notation is used for $C^\alpha_{\text{loc}}(\mathbb{R}^+;E)$. Set $L^{2}_{\sigma}=\{u\in L^2; \div u=0\}$. We denote by $L^{2}_{\rm{loc}}$ the space of locally $L^{2}$-integrable vector fields.
We use $(\Delta_{i})_{i\geq -1}$ to denote the Littlewood--Paley blocks corresponding to a dyadic partition of unity.
Besov spaces on the torus with general indices $\alpha\in \R$, $p,q\in[1,\infty]$ are defined as
the completion of $C^\infty_c(\mathbb{R}^{d})$ with respect to the norm
$$
\|u\|_{B^\alpha_{p,q}}:=\left(\sum_{j\geq-1}2^{j\alpha q}\|\Delta_ju\|_{L^p}^q\right)^{1/q},$$
where $C^\infty_c(\mathbb{R}^{d})$ means smooth functions on $\mR^d$ with compact support. Set $H^\alpha=B^\alpha_{2,2}$ for $\alpha\in\mR$.

Paraproducts were introduced by Bony in \cite{Bon81} and they permit to decompose a product of two distributions into three parts which behave differently in terms of regularity. More precisely, using the Littlewood-Paley blocks, the product $fg$ of two Schwartz distributions $f,g\in\mathcal{S}'(\R^{d})$ can be formally decomposed as
$$fg=f\prec g+f\circ g+f\succ g,$$
with $$f\prec g=g\succ f=\sum_{j\geq-1}\sum_{i<j-1}\Delta_if\Delta_jg, \quad f\circ g=\sum_{|i-j|\leq1}\Delta_if\Delta_jg.$$
Here, the paraproducts $\prec$ and $\succ$ are always well-defined and critical term is the resonant product denoted by $\circ$.
In general, it is only well-defined provided  the sum of the regularities of $f$ and $g$ in terms of Besov spaces is strictly positive.
Moreover, we have the following paraproduct estimates from \cite{Bon81} (see also \cite[Lemma~2.1]{GIP15}).

\begin{lemma}\label{lem:para}
	Let  $\beta\in\R$, $p, p_1, p_2, q\in [1,\infty]$ such that $\frac{1}{p}=\frac{1}{p_1}+\frac{1}{p_2}$. Then  it holds
	\begin{equation*}
		\|f\prec g\|_{B^\beta_{p,q}}\lesssim\|f\|_{L^{p_1}}\|g\|_{B^{\beta}_{p_2,q}},
	\end{equation*}
	and if $\alpha<0$ then
	\begin{equation*}
		\|f\prec g\|_{B^{\alpha+\beta}_{p,q}}\lesssim\|f\|_{B^{\alpha}_{p_1,q}}\|g\|_{B^{\beta}_{p_2,q}}.
	\end{equation*}
	If  $\alpha+\beta>0$ then it holds
	\begin{equation*}
		\|f\circ g\|_{B^{\alpha+\beta}_{p,q}}\lesssim\|f\|_{B^{\alpha}_{p_1,q}}\|g\|_{B^{\beta}_{p_2,q}}.
	\end{equation*}
\end{lemma}

\subsection{Probabilistic elements}\label{s:p}
We will present the probabilistic extension of solutions in Section~\ref{sec:3.2} in a general framework. To this end, we introduce basics for the general stochastic forcing, which will be used there.
For a Hilbert space $\mU$, let $L_2(\mU,L^2_\sigma)$ be the space all Hilbert--Schmidt operators from $\mU$ to $L^2_\sigma$ with the norm $\|\cdot\|_{L_2(\mU,L^2_\sigma)}$. Let $G: L^2_\sigma\rightarrow L_2(\mU,L^2_\sigma)$ be $\mathcal{B}(L^2_\sigma)/\mathcal{B}(L_2(\mU,L^2_\sigma))$ measurable.
In the following, we assume
$$\|G(x)\|_{L_2(\mU,L_{\sigma}^2)}\leq C(1+\|x\|_{L^2}),$$
for every $x\in C_c^\infty\cap L^2_{\sigma}$ and if in addition $y_n\rightarrow y$ in $L^2_{\text{loc}}$ then we require
$$\lim_{n\rightarrow \infty}\|G(y_n)^*x-G(y)^*x\|_{\mU}=0,$$
where the asterisk denotes the adjoint operator.

Suppose there is another  Hilbert space $\mU_1$ such that the embedding  $\mU\subset  \mU_1$ is Hilbert--Schmidt. We also use $H^{-3}_{\text{loc}}$ to denote the space of distributions  equipped with THE topology given by the seminorms
$$\|g\|_{H^{-3}_{R}}=\sup\{\<g,v\>, v\in C^\infty_c, \|v\|_{H^3}\leq 1, \supp v\subset B_R\},\quad 0<R<\infty,$$
with $B_R=\{x:|x|<R\}$.
 Let  $\Omega_0:=C([0,\infty);H^{-3}_{\text{loc}}\times \mU_1)\cap L^2_{\mathrm{loc}}([0,\infty);L^2_\sigma\times \mU_1)$ and let $\mathscr{P}({\Omega}_0)$ denote the set of all probability measures on $({\Omega}_0,{\mathcal{B}})$ with ${\mathcal{B}}$   being the Borel $\sigma$-algebra coming from the topology of locally uniform convergence on $\Omega_0$. Let  $(x,y):{\Omega_0}\rightarrow C([0,\infty);H_{\rm loc}^{-3}\times \mU_{1})$ denote the canonical process on ${\Omega}_0$ given by
$$(x_t(\omega),y_t(\omega))=\omega(t).$$
For $t\geq 0$ we define   $\sigma$-algebra ${\mathcal{B}}^{t}=\sigma\{ (x(s),y(s)),s\geq t\}$.
Finally,  we define the canonical filtration  ${\mathcal{B}}_t^0:=\sigma\{ (x(s),y(s)),s\leq t\}$, $t\geq0$, as well as its right-continuous version ${\mathcal{B}}_t:=\cap_{s>t}{\mathcal{B}}^0_s$, $t\geq 0$.

\subsection{Useful results from \cite{ABC22}}\label{s:2.3}

A recent breakthrough result from \cite{ABC22} is the non-uniqueness of Leray--Hopf solutions to the following forced Navier--Stokes system:
\begin{align}\label{eq:fns}
\partial_t u&=\Delta u+\bar f- u\cdot \nabla u+\nabla p,\quad \div u=0,
	\\u(0)&=0.\nonumber
\end{align}
More precisely, there exist $T>0$ and $\bar f\in L^1(0,T;L^2)$ and two distinct Leray--Hopf solutions on $[0,T]\times \mR^3$ to the Navier--Stokes system \eqref{eq:fns} with force $\bar f$ and initial data $0$.

\rmk{Before proceeding, we recall the main ideas of the proof in \cite{ABC22}, which also serves as the basis for construction of non-unique solutions to the random PDE \eqref{eql1} in Section \ref{sec:loc} below. In \cite{ABC22}, the authors first considered the Navier--Stokes equations in similarity variables $(\xi,\tau)$, and identified a pair of a steady-state equilibrium and a force $(\bar{U}, \bar{F})$ such that, when the equations are linearized around $\bar{U}$, the resulting linear operator $L_{ss}$ (see \eqref{def:Lss} below) possesses an unstable eigenvalue. They calculated the decay rate of $U_{lin}$ (see \eqref{r:Ulin}), which is a solution to the linear PDE $\partial_\tau U = L_{ss}U$. Subsequently, they constructed a non-trivial trajectory $U^{per}$ on the unstable manifold associated with the most unstable eigenvalue. This led them to conclude that $\bar{u}$ and $\bar{u} + u^{lin} + u^{per}$, expressed in physical variables, are two Leray--Hopf solutions to the original Navier--Stokes equations with force $\bar{f}$.}

To be more precise, we  recall  the following similarity variables from \cite{ABC22}:
\begin{align}
	\xi=\frac{x}{\sqrt t}&,\quad \tau=\log t,\no
	\\ u(t,x)=\frac1{\sqrt t}U(\tau,\xi),&\quad \bar f(t,x)=\frac1{t^{3/2}}\bar F(\tau,\xi). \label{tr:1}
\end{align}
In these variables, the forced Navier--Stokes system \eqref{eq:fns} becomes
\begin{align}\label{e:U}
	\partial_\tau U-\frac12(1+\xi\cdot\nabla)U-\Delta U+U\cdot \nabla U+\nabla P=\bar F,\quad \div U=0.
\end{align}

Suppose that $\bar U$ is the linearly unstable solution for the dynamics of \eqref{e:U} obtained in \cite[Theorem 1.3]{ABC22}, that is, there exists an unstable eigenvalue for the linearized operator $L_{ss}$ defined by
\begin{align}\label{def:Lss}
	-L_{ss}U=-\frac12(1+\xi\cdot\nabla)U-\Delta U+\mathbb{P}(\bar U\cdot \nabla U+U\cdot \nabla \bar U),
\end{align}
where $\mathbb{P}$ is the Leray projector. Set
\begin{align}
	\bar F:=-\frac12(1+\xi\cdot\nabla)\bar U-\Delta \bar U+\bar U\cdot \nabla \bar U.
\end{align}
Using \cite[Theorem 1.3]{ABC22} we know that $\bar u(t,x)=\frac1{\sqrt{t}}\bar U(\xi)$ is a Leray--Hopf solution to the forced Navier--Stokes equations \eqref{eq:fns} with $\bar f(t,x)=\frac1{t^{3/2}}\bar F(\tau,\xi)$.
By \cite[Theorem 4.1]{ABC22}, the linear operator $L_{ss}:D(L_{ss})\subset L^2_\sigma\to L^2_\sigma$ with $D(L_{ss}):=\{U\in L^2_\sigma:U\in H^2(\mR^3),\xi\cdot \nabla U\in L^2(\mR^3)\}$, has an unstable eigenvalue $\lambda$. Then $\lambda$ can be chosen to be maximally unstable, that is
 \begin{align}\label{def:a}
a:=	\text{Re}\lambda=\sup_{z\in \sigma(L_{ss})}\text{Re} z>0,
\end{align}
with $\sigma(L_{ss})$ being the spectrum of the operator $L_{ss}$.
Let  $\eta\in H^k(\mR^3)$ be a  non-trivial smooth eigenfunction for all $k\geq0$. Define
\begin{align}\label{def:Ulin}
	U^{lin}(\tau):=\text{Re}(e^{\lambda\tau} \eta),\end{align}
which is a solution to the linearized PDE
$$\partial_\tau U^{lin}=L_{ss}U^{lin}.$$
We also recall from \cite[(4.16)]{ABC22} that
\begin{align}\label{r:Ulin}
	\|U^{lin}\|_{H^k}=C(k)e^{a\tau},\quad \tau\in\mR, \ k\in\mN.
\end{align}

We also recall the following result from \cite[Lemma 4.4]{ABC22}, which will be used frequently in the sequel.

\bl\label{lem:Lss} For any $\sigma_2\geq \sigma_1\geq0$ and $\delta>0$, it holds
\begin{align*}
	\|e^{\tau L_{ss}}U_0\|_{H^{\sigma_2}}\lesssim \tau^{-(\sigma_2-\sigma_1)/2}e^{\tau(a+\delta)}\|U_0\|_{H^{\sigma_1}},\quad \tau>0,
\end{align*}
for any $U_0\in L^2_\sigma\cap H^{\sigma_1}$. \rmk{Here the implicit constant depends on $\sigma_1,\sigma_2,\delta$.}
\el

\rmk{In the following, we extend the approach of \cite{ABC22} to the stochastic case. Specifically, we first transform the stochastically forced Navier--Stokes system into a forced Navier--Stokes system with random viscosity. In this setting, we also use the steady state $\bar U$ and the linear operator as in \cite{ABC22}. Since the equation now contains a stochastic term, the new forcing term ($\bar H$ below) also becomes random. Following the method in \cite{ABC22}, we aim to construct a non-trivial trajectory $U^{per}$ on the unstable manifold. To achieve this, we treat the random viscosity term as a perturbation and apply a fixed-point argument in a suitable Besov space.}

\section{Linear multiplicative noise}\label{s:3}

In this section, we  prove non-uniqueness of Leray--Hopf solutions to the following Navier--Stokes equations driven by linear multiplicative noise
\begin{align} \label{eql}\dif u+\div (u\otimes u)\dif t+\nabla p\dif t=\Delta u\dif t+u\dif W+f\dif t,\quad \div u=0,
\end{align}
where $W$ is a one-dimensional Brownian motion on stochastic basis $(\Omega,\cF,(\cF_t)_{t\geq0},\bP)$ and  $(\mathcal{F}_t)_{t\geq0}$ is the normal filtration generated by $W$, that is, the  canonical right-continuous filtration augmented by all the $\mathbf{P}$-negligible sets.

\subsection{Construction of non-unique local-in-time solutions}\label{sec:loc}

In the first step, we introduce a new variable $v$ so that the stochastic forced Navier--Stokes system \eqref{eql} rewrites as a forced Navier--Stokes system with a random and time dependent viscosity. This transformation is different from what is usually used in case of linear multiplicative noise, see e.g. \cite{HZZ19}.
In particular, we  define
\begin{align}\label{tr}
	w(t):=u(t)e^{-(W(t)-t/2)}.
	\end{align}
By It\^o's formula we find that
$$\dif e^{-(W(t)-t/2)}=-e^{-(W(t)-t/2)}\dif W+e^{-(W(t)-t/2)}\dif t,$$
which implies that
$$\dif w=e^{-(W(t)-t/2)}\dif u +u \dif e^{-(W(t)-t/2)}+\dif \langle u, e^{-(W(t)-t/2)}\rangle.$$
Thus $w$ solves
$$\p_tw+e^{W-t/2}\div (w\otimes w)-\Delta w+e^{-(W-t/2)}\nabla p=e^{-(W-t/2)}f,\quad \div w=0.$$
In this equation, we already eliminated the stochastic integral but the factor in front of the nonlinear term is not so convenient. \rmk{To this end, we define  the time rescaling of $w$:
\begin{align}\label{tr1}
	 w(t)=v\Big(\int_0^te^{W(s)-s/2}\dif s\Big)=:(v\circ \theta)(t).
	\end{align}}
We have $$\partial_tw(t)=(\partial_tv(\theta(t)))e^{W(t)-t/2}$$
which leads to $$\partial_tv(\theta(t))=e^{-W(t)+t/2}\partial_tw(t)=-\div (w\otimes w)-e^{-W(t)+t/2}\Delta w+e^{-2W(t)+t}\nabla p+e^{-2W(t)+t}f.$$
We observe that $\theta$ is  continuous and strictly increasing and so there exists an inverse $\theta^{-1}:\mR^+\to\mR^+$. Thus, we obtain
 that $v$ satisfies
\begin{align} \label{eql1}
\partial_tv+\div (v\otimes v)+\nabla \pi =h(\theta^{-1}(t))\Delta v+g,\quad \div v=0,
\end{align}
where
$$h(t)=e^{-W(t)+t/2},\quad g(t)=h^2(\theta^{-1}(t))f(\theta^{-1}(t)),\quad \pi(t)=h^2(\theta^{-1}(t))p(\theta^{-1}(t)).$$
In other words, the random and time dependent factor now  appears in \eqref{eql1} in place of viscosity. This is very helpful because the dissipative term will be treated as a perturbation.

However, due to the the above random time change we need to carefully trace adaptedness, in order to guarantee that the final solutions to \eqref{eql} are adapted to $(\mathcal{F}_{t})_{t\geq0}$. This turns out to be rather subtle. To this end, we first adjust the filtration and show several auxiliary results on adaptedness of stochastic processes and on stopping times. We also emphasize that the adaptedness to $(\mathcal{F}_{t})_{t\geq0}$ is important for the extension of the solutions in a probabilistic sense (see Section \ref{sec:3.2}),  which stands out as one of the main distinctions in our proof compared to deterministic calculations.

\bl\label{lem:1}
 For every $s\geq0$, $\theta^{-1}(s)$ is an $(\cF_t)$-stopping time. Define
$$\hat{\cF}_t:=\cF_{\theta^{-1}(t)}:=\{A\in \cF;\,A\cap \{\theta^{-1}(t)< s\}\in\cF_{s} \text{ for all }s\geq 0\},\quad t\geq0.$$
Then $(\hat{\cF}_t)_{t\geq0}$ is  a right-continuous filtration. Moreover, if a stochastic process $X$ is  $(\mathcal{F}_{t})$-progressively measurable then $X\circ\theta^{-1}$ is $(\hat{\cF}_t)$-adapted.
\el

\begin{proof}
	For any $s,t\geq0$, $\{\theta^{-1}(s)\leq t\}=\{\theta(t)\geq s\}\in \cF_t$. Hence, the first result follows. Since $\{\theta^{-1}_t,t\geq0\}$ is a family of increasing $(\cF_t)$-stopping times, \rmk{by \cite[Lemma 2.15]{KS}}  we know that the $\sigma$-algebras $\hat\cF_t$, $t\geq 0$, are also increasing, hence they build a filtration.
	For $A\in \cap_{n=1}^\infty \hat\cF_{t+\frac1n}$, by the continuity and monotonicity of $\theta^{-1}$ we have for any $u\geq0$
	$$A\cap \{\theta^{-1}(t)< u\}=\bigcup_{n}A\cap \big\{\theta^{-1}(t+\frac1n)< u\big\}\in \cF_u.$$
	Hence, $A\in \hat\cF_{t}$ and $(\hat{\cF}_t)_{t\geq0}$ is then a right-continuous filtration. The last result follows by a well-known result for stopping times \rmk{(c.f. \cite[Proposition 2.18]{KS})}.
\end{proof}

\bl\label{lem:2} Suppose that $v$ is an $(\hat{\mathcal{F}}_{t})$-adapted $H^{\gamma}$-valued stochastic process with $\mathbf{P}$-a.s. continuous trajectories for some $\gamma\in \R$. Then $v\circ\theta $ is $(\mathcal{F}_{t})$-adapted.
\el

\begin{proof}
	Using the continuity of $v$ with respect to  $t$, we find
	\begin{align*}
		v(\theta(t))=\lim_{n\to\infty}\sum_{k=0}^\infty v\Big(\frac{k}{2^n}\Big)\1_{\{\frac{k}{2^n}\leq \theta(t)<\frac{k+1}{2^n}\}}.
	\end{align*}
Now it suffices to prove that $v(\frac{k}{2^n})\1_{\{\frac{k}{2^n}\leq \theta(t)<\frac{k+1}{2^n}\}}\in \cF_t$.
	For any open set $A\subset H^{{\gamma}}$ \rmk{not including zero}
	\begin{align*}
		\Big\{v\Big(\frac{k}{2^n}\Big)\1_{\{\frac{k}{2^n}\leq \theta(t)<\frac{k+1}{2^n}\}}\in A\Big\}
		&=\Big\{v\Big(\frac{k}{2^n}\Big)\in A\Big\}\cap \Big\{\frac{k}{2^n}\leq\theta(t)<\frac{k+1}{2^n}\Big\}
		\\&=\Big\{v\Big(\frac{k}{2^n}\Big)\in A\Big\}\cap \Big\{\theta^{-1}\Big(\frac{k}{2^n}\Big)\leq t\Big\}\cap \Big\{\theta(t)<\frac{k+1}{2^n}\Big\}\in \cF_t,
	\end{align*}
where we used that $\{v(\frac{k}{2^n})\in A\}\in \hat\cF_{\frac{k}{2^n}}=\cF_{\theta^{-1}(\frac{k}{2^n})}$. \rmk{For any open set $A$ including zero,
\begin{align*}&\Big\{v\Big(\frac{k}{2^n}\Big)\1_{\{\frac{k}{2^n}\leq \theta(t)<\frac{k+1}{2^n}\}}\in A\Big\}
\\=&\Big(\Big\{v\Big(\frac{k}{2^n}\Big)\in A\Big\}\cap \Big\{\frac{k}{2^n}\leq\theta(t)<\frac{k+1}{2^n}\Big\}\Big) \cup \Big\{\frac{k}{2^n}\leq\theta(t)<\frac{k+1}{2^n}\Big\}^c\in \cF_t. \end{align*}}
Hence, the result follows.
\end{proof}

We also prove the following results for stopping time.

\bl\label{lem:3} Suppose that $T$ is an $(\hat\cF_t)$-stopping time. Then $\theta^{-1}\circ T$ is an $(\cF_t)$-stopping time.
\el
\begin{proof}
	We have
	\begin{align*}\{\theta^{-1}\circ T< t\}=\{T< \theta(t)\}&=\bigcup_{s\in\mQ}\{T< s\}\cap\{s< \theta(t)\}
		=\bigcup_{s\in\mQ}\{T< s\}\cap\{\theta^{-1}(s)< t\}\in \cF_t.\end{align*}
\end{proof}

 By Lemma \ref{lem:1} we find that the random coefficient $h\circ\theta^{-1} $ in  \eqref{eql1} is $(\hat\cF_t)$-adapted. Our aim is to find two $(\hat \cF_t)$-adapted Leray--Hopf solutions to   \eqref{eql1} with the same $(\hat \cF_t)$-adapted force $g$  before some strictly positive $(\hat \cF_t)$-stopping time $\hat T$. Then using Lemma \ref{lem:2} and Lemma \ref{lem:3} and the transform in \eqref{tr}, we construct two $( \cF_t)$-adapted Leray--Hopf solutions to   \eqref{eql} with the same $( \cF_t)$-adapted force $f$  before some strictly positive $( \cF_t)$-stopping time $T_{0}=\theta^{-1}\circ \hat T$.

As the next step,  we concentrate on \eqref{eql1}, and find the force $g$ which gives raise to two Leray--Hopf solutions. The transform in the similarity variables as in \eqref{tr:1}, i.e.
\begin{align}\label{tr2}
	v(t,x)=\frac1{\sqrt t}U(\tau,\xi), \quad g(t,x)=\frac1{t^{3/2}}{H}(\tau,\xi),
\end{align}
leads to
\begin{align} \label{eql2}\partial_\tau U-\frac12(1+\xi\cdot\nabla)U+U\cdot \nabla U+\nabla P=h(\theta^{-1}(e^\tau))\Delta U+{H},\quad \div U=0.
\end{align}
This has a similar structure as \eqref{e:U} except for the viscosity.
Then the background solution $\bar U$ defined   in Section \ref{sec:2} based on the construction from  \cite[Theorem 1.3]{ABC22} is  a solution to \eqref{eql2} with $H=\bar H$ given by
 $$ \bar H=-\frac12(1+\xi\cdot\nabla)\bar U+\bar U\cdot \nabla \bar U-h(\theta^{-1}(e^\tau))\Delta \bar U.$$
As $\bar U$ is deterministic, using Lemma \ref{lem:1} and the transform \eqref{tr2} it follows that
\begin{align}\label{def:g}
	 g(t,x)=\frac1{t^{3/2}}\bar H(\tau,\xi)
\end{align} is $(\hat{\cF}_t)$-adapted. Hence, using Lemma \ref{lem:2}
\begin{align}\label{def:f}
	f(t,x):=h(t,x)^{-2}g(\theta(t))
\end{align}
is $({\cF}_t)$-adapted, where $h(t,x)^{-2}=1/h(t,x)^2$. Note, however, that here  $g$ is not continuous at $t=0$. But we can still apply Lemma \ref{lem:2} to the function $\tilde{g}(t,x)=t^{3/2}g(t,x)$ which is continuous to deduce that $\tilde{g}(\theta)$ is  $({\cF}_t)$-adapted. If $t=0$ then $f$ is deterministic and hence measurable with respect to  $\cF_0$, which implies the adaptedness of $f$.

In the following, we fix the above $H=\bar H$, which belongs to  $C([\tau_1,\tau_2];L^2)$ for any $\tau_1,\tau_2\in\mR $ $\bP$-a.s.  and we aim to construct the second solution to \eqref{eql2}.  Similarly to  \cite{ABC22}, we make use of the following ansatz for the second solution
\begin{align}\label{defU}
	U=\bar U+U^{lin}+U^{per},
	\end{align}
with $\bar U$ and $U^{lin}$ as in Section \ref{sec:2} and $U^{per}$ solves the following equation
\begin{align}\label{eqper1}
&\partial_\tau U^{per}-L_{ss}U^{per}+\mathbb{P}\Big(U^{lin}\cdot \nabla  U^{per}+U^{per}\cdot \nabla U^{lin}+U^{lin}\cdot \nabla U^{lin}+U^{per}\cdot \nabla U^{per}\Big)\nonumber\\
	&=(h(\theta^{-1}(e^\tau))-1)\Delta (U^{per}+U^{lin}).
\end{align}
Here $L_{ss}$ was defined in \eqref{def:Lss}. For $U^{per}$, we additionally require a suitable decay estimate which compared to \eqref{r:Ulin} implies that $U^{lin}\neq -U^{per}$ and consequently $U\neq \bar U$ leading to non-uniqueness.

Compared to \cite{ABC22}, the only difference in \eqref{eqper1} comes from the right hand side. The idea  is to view it as a perturbation of the equation. To obtain the necessary estimates, we use the Besov space $B^N_{2,\infty}$ (to replace $H^N$ in \cite{ABC22}) in the definition of the following Banach space $X$, i.e. for some $\eps>0$ and $N>5/2$, $N\in\mathbb{N}$ we let
$$X:=\{U\in C((-\infty,T];B^N_{2,\infty}): \|U\|_X<\infty\},$$
with the norm $$\|U\|_X:=\sup_{\tau<T}e^{-(a+\eps)\tau}\|U(\tau)\|_{B^N_{2,\infty}}.$$

By the time change in \eqref{tr:1} we also need to define the following filtration $\cG_\tau:=\hat \cF_{t}$ for $\tau=\log t\in \mR$.
Since the time change is deterministic, it holds  that $v$ is $(\hat \cF_t)$-adapted if and only if $U$ is $( \cG_\tau)$-adapted for $v$, $U$  in \eqref{tr2}. Furthermore,  $T$ is a  $( \cG_\tau)$-stopping time if and only if  $e^T$ is a $(\hat \cF_t)$-stopping time. Or in other words,  $T$ is a $(\hat \cF_t)$-stopping time if and only if  $\log T$ is a $( \cG_\tau)$-stopping time.

In the following, we consider \eqref{eqper1} with the  random coefficient $h(\theta^{-1}(e^\tau))-1$ adapted to the filtration $(\cG_\tau)_{\tau\in\R}$ and we want to find one solution adapted to the same filtration. \rmk{To this end, we first prove the following  bound for the operator $L_{ss}$.}

\bl\label{lem:im}
For $N\in\mathbb{N}$, $j\geq 0$,  $0<\tau<2$ there exists a constant $c>0$ such that
$$
\|\Delta_j \nabla^{N+2}e^{\tau L_{ss}}U_0\|_{L^2}\lesssim 2^{2j}e^{-c2^{2j}\tau}\|U_0\|_{B^N_{2,\infty}}+\|U_0\|_{L^2},
$$
for any $U_0\in B^N_{2,\infty}$.
\el

\begin{proof}
 We use a similar transform as in the proof of \cite[Lemma 4.4]{ABC22}, i.e. we set \rmk{$G(\tau)=e^{\tau L_{ss}}U_0$} and
	$$u(t,x)=\frac1{\sqrt{t+1}}\rmk{G}\Big(\log(t+1),\frac{x}{\sqrt{t+1}}\Big),\quad \bar u(t,x)=\frac1{\sqrt{t+1}}\bar U\Big( \frac{x}{\sqrt{t+1}}\Big).$$
This leads to
	$$\partial_t u-\Delta u=-\mathbb{P}(\bar u\cdot \nabla u+u\cdot \nabla \bar u),\quad u(0)=U_0.$$
	Since $\bar U$ is smooth, we obtain that $\bar u$ is also smooth and in this transform $t=e^\tau-1\simeq \tau$ when $\tau\in(0,2)$. By Lemma \ref{lem:Lss} we know that
	\begin{align}\label{ul2}\sup_{t\in (0,e^2-1]}\|u(t)\|_{L^2}\leq C\|U_0\|_{L^2}.\end{align}
	By the Duhamel formula, the paraproduct decomposition with implicit summation over $i=1,2,3$ and using the smoothness of $\bar u$, we obtain for $t\in (0,e^2-1)$
	\begin{align}\label{estg}&\|\Delta_j \nabla^N u(t)\|_{L^2}\no
		\\&\lesssim e^{-2^{2j}t}\|\Delta_j \nabla^NU_{0}\|_{L^2}+\int_0^te^{-2^{2j}(t-s)}\|\nabla^{N}\Delta_j(\bar u^{i} \prec \partial_{i} u +\bar u^{i} \succ \partial_{i} u +\bar u^{i} \circ \partial_{i} u)\|_{L^2}\dif s \no\\
		&\qquad\qquad+\int_0^te^{-2^{2j}(t-s)}\|\nabla^{N}\Delta_j( u^{i} \prec \partial_{i}\bar u + u^{i} \succ \partial_{i}\bar u + u^{i} \circ \partial_{i} \bar u)\|_{L^2}\dif s \no
		\\&\lesssim e^{-2^{2j}t}\|\Delta_j \nabla^NU_{0}\|_{L^2}+\sum_{l\sim j}\int_0^te^{-2^{2j}(t-s)}2^{2j}\|\nabla^{N-1}\Delta_l u\|_{L^2}\dif s+\sup_{t\in(0,e^2-1)}\|u(t)\|_{L^2}.
	\end{align}
\rmk{Here and in the following $l\sim j$ means that there exists a universal constant $c>0$ such that $|l-j|\leq c$.}
Indeed, in the last step we used paraproduct estimates Lemma \ref{lem:para}, smoothness of $\bar u$ and
	 $$\|\Delta_j\nabla^{N}(\bar u\prec \nabla u)\|_{L^2}\lesssim\sum_{l\sim j}(\|\nabla^{N+1}\Delta_l u\|_{L^2}+\|\nabla\Delta_l u\|_{L^2})\lesssim 2^{2j}\sum_{l\sim j}\|\nabla^{N-1}\Delta_l u\|_{L^2}+2^j\| u\|_{L^2}, $$
	  $$\|\Delta_j\nabla^{N}(\nabla\bar u\prec  u)\|_{L^2}\lesssim\sum_{l\sim j}(\|\nabla^{N}\Delta_l u\|_{L^2}+\|\Delta_l u\|_{L^2})\lesssim 2^{j}\sum_{l\sim j}\|\nabla^{N-1}\Delta_l u\|_{L^2}+\| u\|_{L^2}, $$
	   and we used
 $\sup_{t\in(0,e^2-1)}\|u(t)\|_{L^2}$  to directly control the remaining terms  as $\bar u$ is smooth.
	
	Choosing $N=1$ in \eqref{estg}, the second term on the right hand side gives $\|u\|_{L^2}$ hence  we obtain for $t\in (0,e^2-1)$
	\begin{align}\label{eqest1}\|\Delta_j \nabla u(t)\|_{L^2}&\lesssim e^{-2^{2j}t}\|\Delta_j \nabla U_0\|_{L^2}+\|U_0\|_{L^2}.
	\end{align}
\rmk{Here we used \eqref{ul2} to bound the second and the last term.}
	Choosing $N=2$ in  \eqref{estg}, we apply \eqref{eqest1} to control the second term and obtain for $t\in (0,e^2-1)$
	\begin{align*}\|\Delta_j \nabla^2 u(t)\|_{L^2}&\lesssim e^{-2^{2j}t}\|\Delta_j \nabla^2 U_0\|_{L^2}+\int_0^te^{-2^{2j}(t-s)}e^{-2^{2j}s}2^{2j}\dif s\sum_{l\sim j}\|\Delta_l \nabla U_0\|_{L^2}+\|U_0\|_{L^2}
		\\&\lesssim e^{-2^{2j}t}\|\Delta_j \nabla^2 U_0\|_{L^2}+e^{-c2^{2j}t}\sum_{l\sim j}\|\Delta_l \nabla U_0\|_{L^2}+\|U_0\|_{L^2}.
	\end{align*}
	For a general $N\in\N$, we iterate the above argument to have for $t\in (0,e^2-1)$
	\begin{align*}\|\Delta_j \nabla^N u(t)\|_{L^2}&\lesssim e^{-2^{2j}t}\|\Delta_j \nabla^N U_0\|_{L^2}+\sum_{l\sim j}\int_0^te^{-c2^{2j}t}2^{2j}\dif s\|\Delta_l \nabla^{N-1} U_0\|_{L^2}+\|U_0\|_{L^2}
		\\&\lesssim e^{-2^{2j}t}\|\Delta_j \nabla^N U_0\|_{L^2}+\sum_{l\sim j}e^{-c2^{2j}t}\|\Delta_l \nabla^{N-1} U_0\|_{L^2}+\|U_0\|_{L^2},
	\end{align*}
	which implies the desired result.
\end{proof}

\bl\label{lem:per} For an integer $N>5/2$, there exist $\eps>0$ and a $(\cG_{\tau})$-stopping time $T=\tau_0 \wedge \log\tau_R \in\mR$ \rmk{with $\tau_0$ given in \eqref{tau} below} and a $(\cG_\tau)$-adapted stochastic process $U^{per}\in C((-\infty,T];B^{N}_{2,\infty})$, which is a solution to \eqref{eqper1} and satisfies \rmk{for $k<N$
$$\|U^{per}(\cdot,\tau)\|_{H^k}\lesssim\|U^{per}(\cdot,\tau)\|_{{B^N_{2,\infty}}}\leq e^{(a+\eps)\tau}, \quad \tau\leq T.$$}
Here $\tau_R$ and $\tau_0$ are given in the proof.
\el
\begin{proof} Choose $0<\eps<\frac14$.
	Define the following $(\hat\cF_t)$-stopping time for  an arbitrary $R>0$ $$\tau_R:=\tau_R^1\wedge \tau_R^2,$$
	$$\tau_R^1:=\inf\{t>0, \|W_{\theta^{-1}}\|_{C_t^{1/4}}\geq R\},\qquad \tau_R^2:=\inf\{t>0, |\theta^{-1}(t)|\geq R\}.$$
In the following we consider the time $\tau\leq \log\tau_R$ and intend to apply fix point argument in a small ball of the Banach space $X$ as in \cite{ABC22}.
The mild formulation of   \eqref{eqper1} reads as
\begin{align*}
	U^{per}=\int_{-\infty}^\tau e^{(\tau-s)L_{ss}}(I_1(s)+I_2(s))\dif s,
\end{align*}
with
\begin{equation}\label{de:I}
	\aligned
	I_1&:=-\mathbb{P}\Big(U^{lin}\cdot \nabla  U^{per}+U^{per}\cdot \nabla U^{lin}+U^{lin}\cdot \nabla U^{lin}+U^{per}\cdot \nabla U^{per}\Big),
	\\I_2&:=(h(\theta^{-1}(e^s))-1)\Delta (U^{per}+U^{lin}).
	\endaligned
\end{equation}
We shall particularly focus on $I_2$ which comes from the right hand side of \eqref{eqper1}, whereas $I_{1}$ will be bounded  below using the estimates from \cite[Section~4.2]{ABC22}.
	
	For the second term   in $I_2$ we apply Lemma \ref{lem:Lss} and \eqref{r:Ulin}  to have for $0<\delta<\eps$
	\begin{align*}&\bigg\|\int_{-\infty}^\tau e^{(\tau-s)L_{ss}}(h(\theta^{-1}(e^s))-1)\Delta U^{lin}\dif s\bigg\|_{B^N_{2,\infty}}
		\\&\lesssim\int_{-\infty}^\tau e^{(\tau-s)(a+\delta)}e^{\frac14s}\|U^{lin}(s)\|_{H^{N+2}}\dif s\lesssim\int_{-\infty}^\tau e^{(\tau-s)(a+\delta)}e^{as+\frac14s}\dif s\lesssim e^{a\tau+\frac14\tau}.
	\end{align*}
	Here, we used the definition of the stopping time $\tau_{R}$ to get
	\begin{align}\label{estb}
	|h(\theta^{-1}(e^s))-1|\lesssim_{R} \Big|\frac{\theta^{-1}(e^s)}2-W_{\theta^{-1}(e^s)}\Big|\lesssim_{R} |\theta^{-1}(e^s)|^{1/4}\lesssim_{R} e^{s/4}, \quad e^s\leq \tau_R,
	\end{align}
since  $\theta^{-1}$ has bounded derivatives before $\tau_R$, which was used in the last step.

Next,  we concentrate on the first term in $I_{2}$.
 By Lemma \ref{lem:im}, Lemma \ref{lem:Lss} and \eqref{estb} for $N>5/2, j\geq0$
	\begin{align*}
	&2^{jN}\bigg\|\Delta_j\int_{-\infty}^\tau e^{(\tau-s)L_{ss}}(h(\theta^{-1}(e^s))-1)\Delta U^{per}\dif s\bigg\|_{L^2}
		\\&\lesssim\bigg\|\Delta_j\nabla^N\int_{-\infty}^{\tau} e^{(\tau-s)L_{ss}}(h(\theta^{-1}(e^s))-1)\Delta U^{per}\dif s\bigg\|_{L^2}
		\\&\lesssim\int_{-\infty}^{\tau-2} e^{(\tau-s)(a+\delta)}e^{(a+\eps)s+\frac14s}\dif s\|U^{per}\|_X+\int_{\tau-2}^{\tau} e^{(a+\eps)s+\frac14s}\Big(2^{2j}e^{-2^{2j}(\tau-s)}+1\Big)\dif s\|U^{per}\|_X
		\\&\lesssim e^{(a+\eps)\tau+\frac14\tau}\|U^{per}\|_X,
	\end{align*}
where $(\Delta_j)_{j\geq-1}$ denotes the Littlewood--Paley blocks corresponding to a dyadic partition of unity and we used Lemma \ref{lem:Lss} for $s\in (-\infty,\tau-2]$ and Lemma \ref{lem:im} for $s\in [\tau-2,\tau]$. For $j=-1$ we directly apply Lemma \ref{lem:Lss}.
	Hence, we derive
	\begin{align*}
	&\Big\|\int_{-\infty}^\tau e^{(\tau-s)L_{ss}}(h(\theta^{-1}(e^s))-1)\Delta U^{per}\dif s\Big\|_{B^N_{2,\infty}}\lesssim e^{(a+\eps)\tau+\frac14\tau}\|U^{per}\|_X.
	\end{align*}
Combining the above calculation we obtain
\begin{align*}
	\Big\|\int_{-\infty}^\tau e^{(\tau-s)L_{ss}}I_2(s)\dif s\Big\|_{B^N_{2,\infty}}
	\lesssim e^{(a+\eps)\tau+\frac14\tau}\|U^{per}\|_X+e^{a\tau+\frac14\tau}.
\end{align*}

As mentioned above, we apply the approach of  \cite[Section 4.2]{ABC22} to control $I_{1}$ as follows
\begin{align*}
	\Big\|\int_{-\infty}^\cdot e^{(\cdot-s)L_{ss}}I_1(s)\dif s\Big\|_{X}
	\lesssim e^{(a+\eps)T}\|U^{per}\|_X^2+e^{(a-\eps)T}+e^{aT}\|U^{per}\|_X.
\end{align*}

Altogether, this leads to
\begin{align*}
	&\Big\|\int_{-\infty}^\cdot e^{(\cdot-s)L_{ss}}(I_1(s)+I_2(s))\dif s\Big\|_{X}
	\\&\leq C\Big(e^{(a+\eps)T}\|U^{per}\|_X^2+e^{(a-\eps)T}+e^{aT}\|U^{per}\|_X\Big)+C\Big(e^{(\frac14-\eps)T}+e^{\frac14T}\|U^{per}\|_X\Big).
\end{align*}
	We choose a deterministic $\tau_0$ very negative such that
\begin{align}\label{tau}2C(e^{(a+\eps)\tau_0}+e^{(a-\eps)\tau_0}+e^{a\tau_0})+C(e^{(\frac14-\eps)\tau_0}+e^{\frac14\tau_0})\leq 1/2.\end{align} Then it is standard  to apply the fix point argument as in \cite{ABC22} to find the desired solution in $X$ with $T=\tau_0\wedge \log \tau_R$.
\end{proof}

Going back to the Navier--Stokes equations in the physical variables \eqref{eql}, we deduce the following.

\bt\label{th:2}
There exist an $(\mathcal{F}_t)$-stopping time $T_0>0$ and an $(\cF_t)$-adapted $f\in L^1(0,T_0;L^2)$ $\bP$-a.s. such that there exists two  distinct $(\mathcal{F}_t)$-adapted Leray--Hopf solutions in $L^\infty(0,T_0;L^2)\cap L^2(0,T_0;H^1)\cap C_w([0,T_0];L^2)$ $\bP$-a.s. to the Navier--Stokes equations \eqref{eql} on $[0,T_0]\times\mR^3$  and initial data $u_0\equiv0$, i.e.
for all $t\in [0,T_0]$ and all divergence-free $\psi\in C^\infty_c(\mR^3)$
\begin{align*}
	\<u(t),\psi\>=\int_0^t\<u,\Delta\psi\>\dif r-\int_0^t \<u\cdot \nabla u,\psi\>\dif r+\int_0^t \<f,\psi\>\dif r+\int_0^t\<\psi,u\dif W\>,
\end{align*}
and the  following energy inequality holds true for all $t\in\mR^+$
\begin{align}\label{energy1}\E\|u(t\wedge T_0)\|_{L^2}^2+2\E\int_0^{t\wedge T_0}\|\nabla u\|_{L^2}^2\dif s\leq 2\E\int_0^{t\wedge T_0}\langle f,u\rangle\dif s+\E\int_0^{t\wedge T_0}\|u\|_{L^2}^2\dif s.
\end{align}
Moreover, the mapping $W\mapsto f=f(W)$ is continuous from $C([0,T])$ to $L^1_{T}L^2$.
\et

\begin{proof}
As a consequence of Lemma~\ref{lem:per} and \eqref{r:Ulin}, we deduce that \eqref{eql2} with $H=\bar H$ on $(-\infty,T]$, $T=\tau_{0}\wedge\log\tau_{R}$, admits two $(\mathcal{G}_{\tau})$-adapted solutions $\bar U$ and $U$ defined in \eqref{defU}.
The transform \eqref{tr2} then permits to go back to the physical variables, and there exist two distinct $(\hat \cF_t)$-adapted Leray--Hopf solutions
	$v_{1}$, $v_{2}$ to the Navier--Stokes equations \eqref{eql1} with $g$ given in \eqref{def:g} on $[0,\hat T]$ with $\hat T={e^{\tau_{0}}\wedge\tau_{R}}$  where
	$$v_{1}(t,x)=\frac1{\sqrt t}\bar U(\xi), \quad v_{2}(t,x)=\frac1{\sqrt t} U(\tau,\xi). $$ \rmk{In fact using Lemma~\ref{lem:per} and \eqref{r:Ulin} we see $(U^{lin}+U^{per})(\tau)\neq0$ for $\tau\to-\infty$ which implies $\bar U$ and $U$ are different. Then according to the change of variables formula, this implies $v_1,v_2$ are different when $t\to 0$. }
By Lemma \ref{lem:per}, 	$\hat T$ is a $(\hat\cF_t)$-stopping time.

Finally, we define
	 $$u_1(t)=e^{W(t)-t/2}(v_1\circ \theta)(t),\quad  u_2(t)=e^{W(t)-t/2} (v_2\circ\theta)(t)$$
	which gives two distinct $(\cF_t)$-adapted solutions to the original Navier--Stokes equations \eqref{eql} on $[0,T_{0}]$  with $T_{0}=\theta^{-1}\circ\hat T$ and  $f$ given in \eqref{def:f}.
Indeed, by Lemma \ref{lem:3}, $T_0$ is an $(\cF_t)$-stopping time.
In view of Lemma \ref{lem:2}, both $u_1$ and $u_2$ are $(\cF_t)$-adapted. Based on  change of variables and regularity of $\bar U$, $U$ we obtain
\begin{align}\label{re:u_i}
	\|u_i(t)\|_{L^2}\lesssim \theta(t)^{1/4}\lesssim t^{1/4},\quad	\|\nabla u_i(t)\|_{L^2}\lesssim \theta(t)^{-1/4}\lesssim t^{-1/4},\qquad i=1,2,
	\end{align}
where we used that $\theta(t)\sim t$ before $T_0$. In fact, if $t\leq T_0$ then $\theta(t)\leq \tau_R$ which implies that $|W_t|\leq R$, $|t|\leq R$. Thus   $\theta(t)\sim t$. This implies that $u_i\in C([0,T_0];L^2)\cap L^2(0,T_0;H^1)$ and the integral $\int_0^{T_0} \langle \div(u_{i}\otimes u_{i}),u_{i}\rangle$ is finite. By Lemma \ref{lem:per} and It\^{o}'s formula, we find that $u_1$ and $u_2$ have the desired regularity on $[0,T_0]$ $\bP$-a.s.,   satisfy the Navier--Stokes equations \eqref{eql} in the analytically weak sense with $f$ given in \eqref{def:f},  and also satisfy the energy inequality on $[0,T_0]$. 
\rmk{More precisely,  $w_1(t)=(v_1\circ \theta)(t)$ satisfies
$$\p_tw_1+e^{W-t/2}\div (w_1\otimes w_1)-\Delta w_1+e^{W-t/2}\nabla \pi\circ \theta=e^{W-t/2}g\circ \theta,\quad \div w=0.$$
By It\^o's formula we find 
$$\dif e^{W(t)-t/2}=e^{W(t)-t/2}\dif W,$$
which implies 
\begin{align*}\dif u_1&=e^{W(t)-t/2}\dif w_1+w_1 \dif e^{W(t)-t/2}+\dif \langle w_1, e^{W(t)-t/2}\rangle
\\&=-\div (u_1\otimes u_1)+\Delta u_1-e^{2W-t}\nabla \pi\circ \theta+f+u_1\dif W.\end{align*}
For $u_2$ we have a similar calculation. Also $u_1(0)=u_2(0)=v_1(0)=v_2(0)=0$. Thus $u_1$ and $u_2$   satisfy the Navier--Stokes equations \eqref{eql} in the analytically weak sense with $f$ given in \eqref{def:f}.}
From Lemma~\ref{lem:per} and \eqref{r:Ulin} we find that $u_1$ and $u_2$ are different. By \eqref{def:f}, continuity of $\theta$ with respect to $W$ and uniform integrability we know that $W\mapsto f$ is continuous functional with respect to $W$ from $C([0,T])$ to $L^1_{T}L^2$. The result follows.
\end{proof}

From \eqref{re:u_i} and It\^o's formula we further obtain that the above solutions satisfy for any $q\geq1$
	$$
	\E\left(\sup_{r\in [0,t\wedge T_0]}\|u(r)\|_{L^2}^{2q}+\int_{0}^{t\wedge T_0}\|\nabla u(r)\|^2_{L^2}\dif r\right)\lesssim\E\left(\int_0^{t\wedge T_{0}}\|f(r)\|_{L^2}\dif r\right)^{2q}+1<\infty.$$

\br\label{r:7}
As in \cite[(1.24)]{ABC22}, we obtain that  $f\in L^1_TL^p$ for any $p<3$. On the other hand, by a similar argument as in the deterministic setting, we know that if $f\in L^1_TL^3$ then there exists at most one solution in $C_TL^3$ to \eqref{eq1} starting from $0$. The existence of such a solution is guaranteed  locally in time. For the equation of $w$, a similar argument  as in the deterministic case \cite[Chapter 27]{LR02} yields uniqueness of solutions to \eqref{eq1} in $C_TL^3$ when $f\in L^1_TL^3$.
\er

\subsection{Non-uniqueness in law}\label{sec:3.2}

The aim of this section is to perform a probabilistic extension of  the local solutions to global ones  and to establish Theorem \ref{th:2intro}. The extension procedure follows from similar arguments as in \cite{HZZ19}. Since this method can also be applied to more general settings, in this section
we start with  the following general stochastic Navier--Stokes equations
\begin{equation}
	\label{1}
	\aligned
	\dif u-\Delta u \dif t+\div(u\otimes u)\dif t+\nabla p\,\dif t&=G(u)\dif W+f(W)\dif t,
	\\\div u&=0.
	\endaligned
\end{equation}
In the above,  $W$ is a Wiener process on a stochastic basis $(\Omega,\cF, (\cF_t)_{t\geq 0}, \bP)$, $G(u)\dif W$  represents a stochastic force acting on the fluid with $G$ satisfying the assumptions in Section \ref{s:p} and additionally $f(W)$ is a given random force where $f$ is a continuous map from $C(\R^{+};\mU_{1})$ to $L^{1}_{\rm loc}(\R^{+};L^{2})$ and $f(W)(t)\in\sigma\{W(s);s\leq t\}$. As mentioned in introduction, Leray--Hopf solutions to this system can only be constructed in the probabilistically weak sense (c.f. \cite{DD03,FR08}). Let us first recall the precise definition of these solutions.
\rmk{To introduce the notation of the solutions we work on the canonical space $\Omega_0$ and the canonical process $(x,y)$ introduced in Section \ref{s:p}. }

\bd\label{weak solution}
	Let $s\geq 0$ and $x_{0}\in L^{2}_{\sigma}$, $y_0\in \mU_1$. A probability measure $P\in \mathscr{P}({\Omega_0})$ is  a probabilistically weak Leray--Hopf solution to the Navier--Stokes system \eqref{1}  with the initial value $(x_0,y_0) $ at time $s$ provided
	
	\no(M1) $P(x(t)=x_0, y(t)=y_0,  0\leq t\leq s)=1$, for any $n\in\mathbb{N}$
	$$P\left\{(x,y)\in {\Omega}_0: \int_0^n\|G(x(r))\|_{L_2(\mU;L_2^\sigma)}^2\dif r<+\infty\right\}=1.$$

	\no(M2) Under $P$,  $y$ is a  cylindrical $({\mathcal{B}}_{t})_{t\geq s}$-Wiener process on  $\mU$ starting from $y_0$ at time $s$ and for every $\psi\in C^\infty_c(\mR^3)\cap L^2_\sigma$, and for $t\geq s$
	$$\langle x(t)-x(s),\psi\rangle+\int^t_s\langle \div(x(r)\otimes x(r))-\Delta x(r),\psi\rangle \dif r=\int_s^t \langle \psi, G(x(r))  \dif y(r)\rangle+\int_s^t\<\psi,f(y)(r)\>\dif r.$$

	\no (M3) It holds for all $t\geq s$
	$$E^P\|x(t)\|_{L^2}^2+2E^P\int_s^{t}\|\nabla x\|_{L^2}^2\dif r\leq \|x(s)\|_{L^2}^2+2E^P\int_s^{t}\langle f(y)(r),x(r)\rangle\dif r+E^P\int_s^{t}\|G(x(r))\|_{L^2(\mU,L^2_\sigma)}^2\dif r,$$
	and for any $q\in \mathbb{N}$ there exists a positive real function $t\mapsto C_{t,q}$ such that  for all $t\geq s$
	$$E^P\left(\sup_{r\in [0,t]}\|x(r)\|_{L^2}^{2q}+\int_{s}^t\|\nabla x(r)\|^2_{L^2}\dif r\right)\leq C_{t,q}(\|x_0\|_{L^2}^{2q}+1+E^P\|f(y)\|_{L^1_tL^2}^{2q})<\infty.$$
\ed

For our purposes, we also  require a definition of probabilistically weak solutions defined up to a stopping time $\tau$. To this end, we set
$$
{\Omega}_{\tau}:=\{\omega(\cdot\wedge\tau(\omega));\omega\in {\Omega}_0\}.
$$

\bd\label{weak solution 1}
	Let $s\geq 0$ and $x_{0}\in L^{2}_{\sigma}$, $y_0\in \mU_1 $. Let $\tau\geq s$ be a $({\mathcal{B}}_{t})_{t\geq s}$-stopping time. A probability measure $P\in \mathscr{P}({\Omega}_\tau)$ is  a probabilistically weak solution to the Navier--Stokes system \eqref{1} on $[s,\tau]$ with the initial value $(x_0,y_0) $ at time $s$ provided
	
	\no(M1) $P(x(t)=x_0, y(t)=y_0,  0\leq t\leq s)=1$
	and for any $n\in\mathbb{N}$
	$$P\left\{(x,y)\in \Omega_0: \int_0^{n\wedge \tau}\|G(x(r))\|_{L_2(\mU;L_2^\sigma)}^2\dif r<+\infty\right\}=1.$$

	\no(M2) Under $P$, $\langle y(\cdot\wedge \tau),l\rangle_{\mU}$ is a {continuous square integrable $({\mathcal{B}}_{t})_{t\geq s}$}-martingale  starting from $y_0$ at time $s$ with quadratic variation process given by $(t\wedge \tau-s)\|l\|_{\mU}^2$ for $l\in U$.  For every $\psi\in C^\infty_c(\mathbb{R}^3)\cap L^2_\sigma$, and for $t\geq s$
	\begin{align*}
	&\langle x(t\wedge \tau)-x(s),\psi\rangle+\int^{t\wedge \tau}_s\langle \div(x(r)\otimes x(r))-\Delta x(r),\psi\rangle \dif r
	\\&\qquad=\int_s^{t\wedge\tau} \langle \psi, G(x(r))  \dif y(r)\rangle+\int_s^{t\wedge \tau}\<f(y)(r),\psi\>\dif r.
	\end{align*}

	\no (M3) It holds for all $t\geq s$
	\begin{align*}	
&	E^P\|x(t\wedge \tau)\|_{L^2}^2+2E^P\int_s^{t\wedge \tau}\|\nabla x\|_{L^2}^2\dif r
\\&\qquad\leq \|x(s)\|_{L^2}^2+2E^P\int_s^{t\wedge \tau}\langle f(y)(r),x(r)\rangle\dif r+E^P\int_s^{t\wedge \tau}\|G(x(r))\|_{L^2(\mU,L^2_\sigma)}^2\dif r,
\end{align*}
	and for any $q\in \mathbb{N}$ there exists a positive real function $t\mapsto C_{t,q}$ such that  for all $t\geq s$
	\begin{align*}
	&E^P\left(\sup_{r\in [0,t\wedge\tau]}\|x(r)\|_{L^2}^{2q}+\int_{s}^{t\wedge\tau}\|\nabla x(r)\|^2_{L^2}\dif r\right)
	\\&\qquad\leq C_{t,q}\Big(\|x_0\|_{L^2}^{2q}+E^P\Big(\int_0^{t\wedge\tau}\|f(y(r))\|_{L^2}d r\Big)^{2q}+1\Big)<\infty.
	\end{align*}
\ed

First, we show that probabilistically weak solutions in the sense of Definition \ref{weak solution} exist and are stable with respect to approximations of the initial time and the initial condition.

\begin{theorem}\label{convergence 1}
	For every $(s,x_0,y_0)\in [0,\infty)\times L_{\sigma}^2\times \mU_1 $, there exists  $P\in\mathscr{P}({\Omega}_0)$ which is a probabilistically weak solution to the Navier--Stokes system \eqref{1} starting at time $s$ from the initial condition $(x_0,y_0)$  in the sense of Definition \ref{weak solution}. The set of all such probabilistically weak solutions  with the same implicit constant $C_{t,q}$ in  Definition \ref{weak solution} is denoted by $\mathscr{W}(s,x_0,y_0, C_{t,q})$.
	
	Let $(s_n,x_n,y_n)\rightarrow (s,x_0,y_0)$ in $[0,\infty)\times L_{\sigma}^2\times \mU_1 $ as $n\rightarrow\infty$  and let $P_n\in \mathscr{W}(s_n,x_n,y_n, C_{t,q})$. Then there exists a subsequence $n_k$ such that the sequence $(P_{n_k})_{k\in\mathbb{N}}$  converges weakly to some $P\in\mathscr{W}(s,x_0,y_0, C_{t,q})$.
\end{theorem}

\begin{proof}
	The proof follows  the  lines of the proof of \cite[Theorem 5.1]{HZZ19}. The main difference is the function space for tightness since we are now on the whole space $\R^{3}$ instead of the torus $\mathbb{T}^{3}$. In this case we can use \cite[Lemma 2.7]{MR} to deduce that the family  of probability measures $P_{n}$, $n\in\N$, is tight on
	\begin{align*}
		\Big(C(\mR^+;H^{-3}_{\text{loc}})\cap \big(L^2_{\text{loc}}(\mR^+;H^1),w\big)\cap L^2_{\text{loc}}(\mR^+;L^2_{\text{loc}})\Big)\rmk{\times C(\mR^+;\mU_1)},
	\end{align*}
where $\big(L^2_{\text{loc}}(\mR^+;H^1),w\big)$ denotes  the weak topology on  $L^2_{\rm loc}(\R^{+};H^1)$. The convergence of $f(y)$ follows from continuity of $y\mapsto f(y)$ from $C(\mR^+,\mU_1)\to L^1(0,T;L^2)$ for every $T\geq0$.
\end{proof}

As the next step, we shall extend  probabilistically weak solutions defined up a {$({\mathcal{B}}_{t})_{t\geq 0}$}-stopping time $\tau$ to the whole interval $[0,\infty)$. We denote by ${\mathcal{B}}_{\tau}$ the $\sigma$-field associated with  $\tau$.

We recall the following result from \cite[Proposition 5.2, Proposition 5.3]{HZZ19}.

\bp\label{prop:1 1}
	Let $\tau$ be a bounded $({\mathcal{B}}_{t})_{t\geq0}$-stopping time. Then for every $\omega\in {\Omega}_0$ there exists $Q_{\omega}\in\mathscr{P}({\Omega}_0)$  such that for $\omega\in \{x(\tau)\in L^2_\sigma\}$
	\begin{equation}\label{qomega 1}
		Q_\omega\big(\omega'\in\Omega; (x,y)(t,\omega')=(x,y)(t,\omega) \textrm{ for } 0\leq t\leq \tau(\omega)\big)=1,
	\end{equation}
	and
	\begin{equation}\label{qomega2 1}
		Q_\omega(A)=R_{\tau(\omega),x(\tau(\omega),\omega),y(\tau(\omega),\omega)}(A)\qquad\text{for all}\  A\in \mathcal{B}^{\tau(\omega)}.
	\end{equation}
	where $R_{\tau(\omega),x(\tau(\omega),\omega),y(\tau(\omega),\omega)}\in\mathscr{P}({\Omega}_0)$ is a probabilistically weak solution to the Navier--Stokes system \eqref{1} starting at time $\tau(\omega)$ from the initial condition $(x(\tau(\omega),\omega), y(\tau(\omega),\omega))$. Furthermore, for every $B\in{\mathcal{B}}$ the mapping $\omega\mapsto Q_{\omega}(B)$ is ${\mathcal{B}}_{\tau}$-measurable.
\ep

\bp\label{prop:2 1}
	Let $x_{0}\in L^{2}_{\sigma}$.
	Let $P$ be a probabilistically weak solution to the Navier--Stokes system \eqref{1} on $[0,\tau]$ starting at the time $0$ from the initial condition $(x_{0},0)$. In addition to the assumptions of Proposition \ref{prop:1 1}, suppose that there exists a Borel  set $\mathcal{N}\subset{\Omega}_{\tau}$ such that $P(\mathcal{N})=0$ and for every $\omega\in \mathcal{N}^{c}$ it holds
	\begin{equation}\label{Q1 1}
		\aligned
		&Q_\omega\big(\omega'\in\Omega_0; \tau(\omega')=
		\tau(\omega)\big)=1.
		\endaligned
	\end{equation}
	Then the  probability measure $ P\otimes_{\tau}R\in \mathscr{P}({\Omega}_0)$ defined by
	\begin{equation*}
		P\otimes_{\tau}R(\cdot):=\int_{{\Omega}}Q_{\omega} (\cdot)\,P(\dif\omega)
	\end{equation*}
	satisfies $P\otimes_{\tau}R= P$ on $\sigma\{x(t\wedge\tau),y(t\wedge \tau),t\geq0\}$ and
	is a probabilistically weak solution to the Navier--Stokes system \eqref{1} on $[0,\infty)$ with initial condition $(x_{0},0)$.
\ep
\begin{proof}
	The proof follows the  lines of \cite[Proposition 5.3]{HZZ19}. The main difference is that we have to verify  (M3) holds for $P\otimes_{\tau}R$. We have
	 \begin{equation*}
		\aligned
		&E^{P\otimes_{\tau}R}\Big(\|x(t)\|_{L^2}^{2}+2\int_0^t\|\nabla x(r)\|_{L^2}^2\dif r\Big)\\
		&= E^{P\otimes_{\tau}R}\Big(\|x(t\wedge \tau)\|_{L^2}^{2}+2\int_0^{t\wedge\tau}\|\nabla x(r)\|_{L^2}^2\dif r\Big)
		\\&\qquad+E^{P\otimes_{\tau}R}\Big(\|x(t)\|_{L^2}^{2}-\|x(t\wedge\tau)\|_{L^2}^{2}+2\int_{t\wedge\tau}^t\|\nabla x(r)\|_{L^2}^2\dif r\Big).
		\endaligned
	\end{equation*}
For the first term on the right hand side, we have
\begin{align*}
&E^{P\otimes_{\tau}R}\Big(\|x(t\wedge \tau)\|_{L^2}^{2}+2\int_0^{t\wedge\tau}\|\nabla x(r)\|_{L^2}^2\dif r\Big)
\\	& \leq \|x(0)\|_{L^2}^2+2E^{P}\int_0^{t\wedge \tau}\<f(y),x\>\dif r+E^P\int_0^{t\wedge \tau}\|G(x(r))\|_{L^2(\mU,L^2_\sigma)}^2\dif r.
\end{align*}
For the second term we note that under $Q_\omega$
\begin{align*}
	&E^{Q_\omega}\Big(\|x(t)\|_{L^2}^{2}-\|x(t\wedge\tau(\omega))\|_{L^2}^{2}+2\int_{t\wedge\tau(\omega)}^t\|\nabla x(r)\|_{L^2}^2\dif r\Big)
	\\&\leq 2E^{Q_\omega}\int_{t\wedge \tau(\omega)}^t\<f(y),x\>\dif r+E^{Q_\omega}\int_{t\wedge \tau(\omega)}^t\|G(x(r))\|_{L^2(\mU,L^2_\sigma)}^2\dif r.
\end{align*}
Integrating with respect to $P$ and using \eqref{Q1 1}, we deduce
\begin{align*}
	&E^{P\otimes_{\tau}R}\Big(\|x(t)\|_{L^2}^{2}-\|x(t\wedge\tau)\|_{L^2}^{2}+2\int_{t\wedge\tau}^t\|\nabla x(r)\|_{L^2}^2\dif r\Big)
	\\&\leq 2E^{P\otimes_{\tau}R}\int_{t\wedge \tau}^t\<f(y),x\>\dif r+E^{P\otimes_{\tau}R}\int_{t\wedge \tau}^t\|G(x(r))\|_{L^2(\mU,L^2_\sigma)}^2\dif r.
\end{align*}
Hence, the first inequality in (M3) holds for $P\otimes_{\tau}R$ and the second one is similar.
\end{proof}

From now on, we restrict ourselves to the setting of a linear multiplicative noise as in Section \ref{sec:loc}. In particular, the driving Wiener process is real-valued and consequently $\mU=\mU_{1}=\mathbb{R}$.
Furthermore, we choose $f$ as defined in \eqref{def:f}. By definition, it belongs to $L^{1}_{\rm loc}(\R^{+};L^{2})$ and  in Theorem~\ref{th:2} we only used its restriction to $[0,T_{0}]$. In addition,
$f$ is a continuous functional of the driving Brownian motion $W$, namely,   $W\mapsto f(W)$ is a continuous  map from $C(\mR^+;\mU_1)$ to $L^1_{\text{loc}}(\mR^+;L^2)$ satisfying
$f(y)(t)\in \cB_t^0(y)$ for $y\in C(\mR^+;\mU_1)$ and every $t\geq0$, where $\cB_t^0(y)=\sigma\{y(s),s\leq t\}$. Furthermore, we have for every $t\geq0$ and $q\geq1$
\begin{align}\label{bd:f}
	\E\|f\|_{L_t^1L^2}^{2q}<\infty.
\end{align}
Indeed, since $h(t)^{-1}=e^{W(t)-t/2}$ is an exponential martingale, we have for any $p\geq1$
$$\E \big(\max_{s\in [0,t]}h(s)^{-1}\big)^p\lesssim1.$$
Hence, we consider $g(\theta(t))$.  By definition of $g$, we could write
\begin{align*}
	g(\theta(t))=g_1(\theta(t))+h(t)g_2(\theta(t)),
\end{align*}
for some $g_1=\frac1{t^{3/2}}H_1(\frac{x}{\sqrt t})$ and $g_2=\frac1{t^{3/2}}H_2(\frac{x}{\sqrt t})$ with $H_1, H_2$ smooth functions with compact support.
By change of variables we have
\begin{align*}
	\int_0^t\|g_1(\theta(s))\|_{L^2}\dif s\lesssim \int_0^t|\theta(s)|^{-3/4}\dif s.
\end{align*}
Since
$$\theta(t)=\int_0^te^{W(s)-s/2}\dif s\geq \min_{s\in [0,t]} e^{W(s)-s/2}t,$$
we have
\begin{align*}
	\int_0^t\|g_1(\theta(s))\|_{L^2}\dif s\lesssim \int_0^ts^{-3/4}\dif s \max_{s\in [0,t]}(e^{-\frac34W(s)+\frac{3s}8}).
\end{align*}
Hence, we have for any $p\geq 1$
\begin{align*}
	\E\Big(\int_0^t\|g_1(\theta(s))\|_{L^2}\dif s\Big)^p\lesssim 1.
\end{align*}
Similarly, we have
\begin{align*}
	\E\Big(\int_0^t\|h(s)^{-1}g_2(\theta(s))\|_{L^2}\dif s\Big)^p\lesssim 1.
\end{align*}
Hence, \eqref{bd:f} holds.

Next,  we shall  introduce the stopping time as in Section \ref{sec:loc}, i.e. we  define
$$\bar\theta(t):=\int_0^te^{y(s)-s/2}\dif s,\quad t\geq0,$$
which is also positive for $t>0$,  strictly increasing   and continuous for every $y$. Hence we also have the inverse of $\bar \theta$ denoted as $\bar \theta^{-1}$.
For $n\in\mathbb{N}$, $R>1$  we define
\begin{equation*}
	\aligned\bar\tau^n(\omega)&:=\inf\left\{t> 0, |\bar\theta^{-1}(t,\omega)|>R-\frac{1}{n}\right\}\bigwedge \inf\left\{t>0,\|y(\bar\theta^{-1}(t),\omega)\|_{C_t^{\frac{1}{4}}}>R-\frac{1}{n}\right\}\bigwedge e^{\tau_0},
	\endaligned
\end{equation*}
with $\tau_0$ being the deterministic constant given in the proof of Lemma \ref{lem:per}. Set
$$\bar T^n:=\bar \theta^{-1}\circ \bar\tau^n.$$
Then the sequence $\{\bar T^{n}\}_{n\in\mathbb{N}}$ is nondecreasing and we define
\begin{equation}\label{eq:tauL 1}
	\bar T:=\lim_{n\rightarrow\infty}\bar T^n.
\end{equation}
Without additional regularity of the process $y$, it holds true that $\bar\tau^{n}(\omega)=0$.
By \cite[Lemma~3.5]{HZZ19} and Lemma \ref{lem:3} we obtain that $\bar T^{n}$ is $({\mathcal{B}}_t)_{t\geq0}$-stopping time and consequently also  $\bar T$ is a $({\mathcal{B}}_t)_{t\geq 0}$-stopping time as an increasing limit of stopping times.

Now, we fix   a real-valued Wiener process $W$ defined on a probability space $(\Omega, \mathcal{F},\mathbf{P})$ and we denote by  $(\mathcal{F}_{t})_{t\geq0}$ its normal filtration. On this stochastic basis, we apply Theorem~\ref{th:2} and denote by $u_1$ and $u_2$ the corresponding solution to the Navier--Stokes system \eqref{eql} on $[0,T_0]$, where the stopping time $T_{0}$ is defined in the proof of Theorem~\ref{th:2}.
{We recall that $u_i, i=1,2,$ is adapted with respect to $(\mathcal{F}_{t})_{t\geq0}$.}
 We denote by $P_i$ the law of $(u_i,W)$ and obtain the following result by similar arguments as in the proof of \cite[Proposition 5.4]{HZZ19}.

\bp\label{prop:ext 1}
	The probability measure $P_i, i=1,2,$ is a probabilistically weak solution to the Navier--Stokes system \eqref{eql} on $[0,\bar T]$ in the sense of Definition \ref{weak solution 1}, where $\bar T$ was defined in \eqref{eq:tauL 1}.
\ep

\bp\label{prp:ext2 1}
	The probability measure $P_i\otimes_{\bar T}R, i=1,2,$ is a probabilistically weak solution to the Navier--Stokes system \eqref{eql} on $[0,\infty)$ in the sense of Definition \ref{weak solution}.
\ep

\begin{proof}
	The proof follows from similar argument as in \cite[Proposition 3.8]{HZZ19}. In light of Proposition \ref{prop:1 1} and Proposition \ref{prop:2 1}, it only remains to establish \eqref{Q1 1}.
We know that
	\begin{align*}
		\begin{aligned}
			{P_i}\left(\omega:y(\bar\theta^{-1}(\cdot\wedge \bar T(\omega)))\in C^{{\frac13}}_{\mathrm{loc}}(\mR^+;\mR)\right)=1.
		\end{aligned}
	\end{align*}
	This means that  there exists a ${P_i}$-measurable set $\mathcal{N}\subset \Omega_{\bar T}$ such that ${P_i}(\mathcal{N})=0$ and for $\omega\in \mathcal{N}^c$
	\begin{equation}\label{continuity1}
	y(\bar\theta^{-1}(\cdot\wedge \bar T(\omega)))\in C^{{\frac13}}_{\mathrm{loc}}(\mR^+;\mR).
	\end{equation}
Similar as in \cite[Proposition 3.8]{HZZ19} for all  $\omega\in \mathcal{N}^c\cap \{x(\tau)\in L^2_\sigma\}$
	$$
	Q_\omega\left(\omega'\in\Omega_{0}; y(\bar\theta^{-1})\in C^{{\frac13}}_{\mathrm{loc}}(\mR^+;\mR)\right)=1.
	$$

	As a consequence, for all
	$\omega\in \mathcal{N}^c\cap \{x(\tau)\in L^2_\sigma\}$ there exists a measurable set $N_\omega$ such that $Q_\omega(N_\omega)=0$ and for all $\omega'\in N_\omega^c$ the trajectory
	$t\mapsto y(\bar\theta^{-1})(t,\omega')$ belongs to $ C^{{\frac{1}{3}}}_{\mathrm{loc}}(\mR^+;\mR)$. Therefore, by \eqref{eq:tauL 1} and $\bar\theta^{-1}\in C^1_{\text{loc}}(\mR^+;\mR)$ for all $\omega'\in \Omega$ we obtain that
	$\bar T(\omega')=\widetilde T(\omega')$ for all $\omega'\in N_\omega^c$ where $\widetilde T=\bar\theta^{-1}\circ \widetilde \tau$
	\begin{equation*}
		\widetilde{\tau}(\omega'):=\inf\left\{t\geq 0, |\bar\theta^{-1}(t)|\geq  R\right\}\bigwedge\inf\left\{t\geq 0,\|y(\bar \theta^{-1})\|_{C_t^{1/4}}\geq R\right\}\bigwedge e^{\tau_0}.\end{equation*}
	This implies that for $t>0$
	\begin{equation}\label{mea}\aligned
		\left\{\omega'\in N_\omega^c,\widetilde T(\omega')\leq t\right\}
		&=\left\{\omega'\in N_\omega^c,\widetilde \tau(\omega')\leq \bar\theta(t)\right\}
	\\	&=\left\{\omega'\in N_\omega^c, \sup_{s\in\mathbb{Q},s\leq \bar\theta(t)}
		|\bar\theta^{-1}(s)|\geq R\right\}
		\\&\qquad\bigcup\left\{\omega'\in N_\omega^c, \sup_{s_1\neq s_2\in \mathbb{Q}\cap [0,\bar\theta(t)]}\frac{|(y(\bar \theta^{-1}))(s_1)-(y(\bar \theta^{-1}))(s_2)|}{|s_1-s_2|^{\frac{1}{4}}}\geq R\right\}
		\\&\qquad\bigcup\left\{\omega'\in N_\omega^c, e^{\tau_0}\leq \bar\theta(t)\right\}
		\\&=\left\{\omega'\in N_\omega^c, \sup_{s_1\neq s_2\in \mathbb{Q}\cap [0,t]}\frac{|y(s_1)-y(s_2)|}{|\bar\theta(s_1)-\bar\theta(s_2)|^{\frac{1}{4}}}\geq R\right\}
		\\&\qquad\bigcup\left\{\omega'\in N_\omega^c,t\geq R\right\}\bigcup\left\{\omega'\in N_\omega^c, e^{\tau_0}\leq \bar\theta(t)\right\}
		\\&=: N^c_\omega \cap A_t.\endaligned
	\end{equation}
	Finally, we deduce that for all $\omega\in\mathcal{N}^c\cap \{x(\tau)\in L^2_\sigma\}$ with $P_i(x(\tau)\in L^2_\sigma)=1$
	\begin{equation}\label{Q}
		\aligned
		&Q_\omega\big(\omega'\in\Omega; \bar T (\omega')=\bar T(\omega)\big)=Q_\omega\big(\omega'\in N_\omega^c; \bar T (\omega')=\bar T(\omega)\big)
		\\&\quad=Q_\omega\big(\omega'\in N_\omega^c; \omega'(s)=\omega(s), 0\leq s\leq \bar T(\omega), \bar T (\omega')=\bar T(\omega)\big)=1,
		\endaligned
	\end{equation}
	where we used \eqref{qomega 1} and the fact that (\ref{mea}) implies
	$$\{\omega'\in N_\omega^c; \bar T (\omega')=\bar T(\omega)\}=N_\omega^c\cap (A_{\bar T(\omega)}\backslash (\cup_{n=1}^\infty A_{\bar T(\omega)-\frac1n}))\in N_\omega^c\cap \mathcal{B}_{\bar T(\omega)}^0,$$
	and  $Q_\omega(A_{\bar T(\omega)}\backslash (\cup_{n=1}^\infty A_{\bar T(\omega)-\frac1n}))=1$.
	This verifies the condition \eqref{Q1 1} in Proposition \ref{prop:2 1} and as a consequence ${P_i}\otimes_{\tau_{L}}R$ is a probabilistically weak solution to the Navier--Stokes system \eqref{eq1} on $[0,\infty)$ in the sense of Definition \ref{weak solution}.
\end{proof}

\bt \label{th:1law}
There exists a force $f$, which is a measurable functional of the driving Brownian motion $W$ such that there exist two distinct probabilitically weak Leray--Hopf solutions $\bP_1$ and $\bP_2$ to the Navier--Stokes system \eqref{eql} and
$f\in L^1_{\rm loc}(\mR^+;L^2)$ $\bP_i$-a.s.
\et

\begin{proof} Define $\bP_i=P_i\otimes_{\bar T}R$, $i=1,2$ for $P_i\otimes_{\bar T}R$ in Proposition \ref{prp:ext2 1}.
	Using Theorem \ref{th:2} the laws of these two probabilistically weak solutions are distinct. Indeed, before $\bar T$ we see that the rates with which the  two solutions converge to zero are different which implies law of two solutions are different. In fact, we have
		$$P_1\otimes_{\bar T}R\left(x(t)=\frac{e^{y(t)-t/2}}{\sqrt{\theta(t)}}\bar{U}\bigg(\frac{\cdot}{\sqrt{\theta(t)}}\bigg),\ t\leq \bar T\right)=1,$$
	$$P_2\otimes_{\bar T}R\left(x(t)=\frac{e^{y(t)-t/2}}{\sqrt{\theta(t)}}\bar{U}\bigg(\frac{\cdot}{\sqrt{\theta(t)}}\bigg),\ t\leq \bar T\right)=0.$$	 As a consequence joint non-uniqueness in law, i.e. non-uniqueness of probabilistically weak solutions, holds for the Navier--Stokes system \eqref{eql}.
\end{proof}

\section{Deterministic  force in a dense set}\label{sec:det}

The aim of this section is to prove that for any given $f$ in a suitable function space  the following deterministic forced Navier--Stokes equations  on $[0,1]\times\R^{3}$
\begin{align}\label{eq1}\partial_t u&=\Delta u+\bar f+f- u\cdot \nabla u+\nabla p,\quad \div u=0,
	\\u(0)&=0,\no
\end{align}
admits two Leray--Hopf solutions, where $\bar f$ is the force from Section~\ref{sec:2}.
Note that it is enough to construct these solutions on some time interval $[0,T_{0}]$, $T_{0}>0$, as these can be always extended to $[0,1]$ by a Leray--Hopf solution obtained by the usual argument. As a matter of fact, the solutions obtained in \cite{ABC22} are even suitable Leray--Hopf solutions, that is, they additionally satisfy a local energy inequality (see \eqref{eq:locen} below). Based on the discussion in \cite[Chapter 30]{LR02} suitable Leray--Hopf solutions exist for every initial condition in $L^{2}$. Accordingly, the solutions \`a la \cite{ABC22} can be extended to $[0,1]$ by suitable Leray--Hopf solutions and the resulting solutions remain suitable Leray--Hopf.

Next, we recall the notion of Leray--Hopf solution in this setting.

\bd \label{def:sol} Let $u_0\in L^2$ be a divergence-free vector field, and $f+\bar f\in L^1(0,1;L^2)$. A Leray--Hopf solution to the Navier--Stokes system \eqref{eq1} on $[0,1]\times \mR^3$ with initial data $u_0$ and force $f+\bar f$ is a divergence-free vector field $u\in L^\infty(0,1;L^2)\cap L^2(0,1;H^1)\cap C_w([0,1];L^2)$ such that $u(0)=u_0$ and for all $t\in [0,1]$ and all divergence-free $\psi\in C^\infty_c(\mR^3)$
\begin{align*}
	\<u(t),\psi\>-\<u(0),\psi\>=\int_0^t\<u,\Delta\psi\>\dif r-\int_0^t \<u\cdot \nabla u,\psi\>\dif r+\int_0^t \<f+\bar f,\psi\>\dif r,
\end{align*}
and the  following energy inequality holds true for all $t\in(0,1]$
\begin{align}\label{energy}\|u(t)\|_{L^2}^2+2\int_0^t\|\nabla u\|_{L^2}^2\dif s\leq \|u(0)\|_{L^2}^2+2\int_0^t\langle f+\bar f,u\rangle\dif s.
\end{align}
We say that the solution is suitable provided it satisfies the local energy inequality
\begin{equation}\label{eq:locen}
(\partial_{t}-\Delta)\frac12|u|^{2}+|\nabla u|^{2}+\div\left[\left(\frac12|u|^{2}+p\right)u\right]\leq (f+\bar f)\cdot u
\end{equation}
in the sense of distributions on $(0,1)\times\R^{3}$,
where $p\in L^{1}((0,1)\times\R^{3})$ is the associated pressure.
\ed

As the first step, we  solve the following  equation
\begin{align}\label{eq4}\partial_t u&=\Delta u+f-\mathbb{P}[\bar u\cdot \nabla u+u\cdot \nabla \bar u+u\cdot \nabla u],
	\\u(0)&=0,\nonumber
	\end{align}
where $\bar u(t,x)=\frac1{\sqrt{t}}\bar U(\xi)$ with the notation of  Section \ref{sec:2}.
For an integer $N>5/2$ and $\eps>0$, define a Banach space
$$Y:=\{f\in C((0,1),H^N), \|f\|_{Y}<\infty\},$$
with
$$\|f\|_{Y}:=\sup_{t\in[0,1]} \sum_{k=0}^Nt^{\frac34-a-\eps+\frac{k}2}\|\nabla^k f(t)\|_{L^2},$$
where $a$ is introduced in \eqref{def:a}.
It is easy to see that $C_c^\infty((0,1)\times \mR^3)\subset Y\subset L^1(0,1;L^2)$.

\bp\label{pro:1} Assume that $N>5/2$ is an integer and consider $f\in Y$. Then there exist $T\in\mathbb{R}$ and $u\in C([0,e^T];L^2)\cap L^2(0,e^T;H^1)$ a solution to \eqref{eq4} satisfying for any $p\in[2,\infty)$ and $k\leq N$, $k\in\mN$, $t\in [0,e^T]$
\begin{align*}t^{k/2}\|\nabla^k u(t)\|_{L^p}\lesssim t^{\frac12(\frac3p-1)}.
\end{align*}
\ep
\begin{proof}If we consider \eqref{eq4} directly and try to use fixed point argument we will see a problem coming from $\bar u$.  Instead, we perform the following transform as in \eqref{tr:1}
	$$U(\tau,\xi)=U \left(\log t,\frac{x}{\sqrt{t}}\right)=\sqrt{t}u(t,x), \quad f(t,x)=\frac1{t^{3/2}} F \left(\log t,\frac{x}{\sqrt{t}}\right).$$  Then it follows that  $U$ satisfies the following equations
	\begin{align}\label{eq6}
	\begin{aligned}
		\partial_\tau U&=L_{ss}U+\mathbb{P}(F-U\cdot \nabla U),\\
		U(-\infty)&=0,
		\end{aligned}
		\end{align}
	where $L_{ss}$ was defined  in \eqref{def:Lss}.
	This problem can be solved by a fix point argument. By Duhamel formula, we have
	\begin{align*}U(\tau)=\int_{-\infty}^{\tau} e^{(\tau-s)L_{ss}}\mathbb{P}[F-(U\cdot \nabla U)]\dif s
	\end{align*}
	and  we define the norm $$\|U\|_{X_T}:=\sup_{\tau<T}e^{-(a+\varepsilon)\tau}\|U(\tau)\|_{H^N}.$$
	with some $\eps>0$, $a$ given in \eqref{def:a} and $T\in\R$.
	In view of  Lemma \ref{lem:Lss} we have for $\tau\in \mR$, $0<\delta<\varepsilon$
	\begin{align*}
	\|U(\tau)\|_{H^N}&\lesssim \int_{-\infty}^{\tau} \Big(e^{(\tau-s)(a+\delta)+s(a+\varepsilon)}\|F\|_{X_\tau}+\frac{e^{(\tau-s)(a+\delta)}e^{2s(a+\varepsilon)}}{(\tau-s)^{1/2}}\|U\|_{X_\tau}^2\Big)\dif s
		\\&\lesssim e^{\tau(a+\varepsilon)}\|F\|_{X_\tau}+e^{\tau(2a+2\varepsilon)}\|U\|_{X_\tau}^2.
	\end{align*}
	Then we apply a fixed point argument in a small ball in $X_T$ by choosing $T$ very negative and obtain
	$$\|U\|_{X_T}\lesssim \|F\|_{X_T}.$$ Now,  we find a suitable solution $U$ for \eqref{eq6} provided $\|F\|_{X_T}<\infty$. Define
	$$u(t,x)=\frac{1}{\sqrt{t}}U \left(\log t, \frac{x}{\sqrt{t}}\right)$$ and it is easy to see that $u$ is a solution to \eqref{eq4}.
	Since $f\in Y$, we have $\|F\|_{X_T}<\infty$.
	In fact, we have
	$$F(\tau,\xi)=t^{3/2}f(t,x)=e^{3\tau/2} f(e^\tau,\xi e^{\tau/2}),$$ and $$\nabla^kF(\tau,\xi)=e^{3\tau/2} e^{k\tau/2}\nabla^kf(e^\tau,\xi e^{\tau/2}).$$
	The last claim is obtained by change of variables. The proof is complete.
\end{proof}

Proposition \ref{pro:1} and \cite[Theorem 1.3]{ABC22} imply that $u+\bar u$ is a Leray--Hopf solution to the forced Navier--Stokes equations \eqref{eq1}.

As the next step, we construct another Leray--Hopf solution. First, we consider $U(\xi,\tau)=\sqrt{t}u(t,x)$ and observe that by the proof of Proposition \ref{pro:1} it holds that
\begin{align}\label{est}
\|U(\tau)\|_{H^N}\lesssim e^{(a+\eps)\tau},\quad \tau\in(-\infty, T].
\end{align}
Next, we make the following ansatz for the second Leray--Hopf solution:
$$\tilde{U}=\bar U+U+U^{lin}+U^{per},$$
where $U^{lin}$ is defined in \eqref{def:Ulin}.
Consequently,  $U^{per}$ shall  satisfy
\begin{align}\label{eqper}
&\partial_\tau U^{per}-L_{ss}U^{per}+\mathbb{P}\Big(( U+U^{lin})\cdot \nabla  U^{per}+U^{per}\cdot \nabla ( U+U^{lin})\nonumber
\\&+U\cdot \nabla U^{lin} +U^{lin}\cdot \nabla U +U^{lin}\cdot \nabla U^{lin}+U^{per}\cdot \nabla U^{per}\Big)=0.
\end{align}
The latter problem can be solved by a  similar argument as in \cite[Proposition 4.5]{ABC22} as follows.

\bp\label{pro:2} Assume $N>5/2$ is an integer. Then there exist $T\in\mathbb{R}$, $0<\eps_0<a$ and $U^{per}\in C((-\infty,T];H^N)$ a solution to \eqref{eqper} such that
$$\|U^{per}(\tau)\|_{H^N}\leq e^{(a+\eps_0)\tau}, \quad \tau\leq T.$$
\ep
\begin{proof} We apply a fixed point argument in Banach space
$$X:=\{U\in C((-\infty,T];H^N): \|U\|_X<\infty\},$$
with the norm
$$\|U\|_X:=\sup_{\tau<T}e^{-(a+\eps_0)\tau}\|U(\tau)\|_{H^N},$$
with $\eps_0>\delta$ in order to guarantee convergence of a time integral in \eqref{time int} below. By the proof of  \cite[Proposition 4.5]{ABC22},
we know that
\begin{align*}&
	\Big\|\int_{-\infty}^{\tau}e^{(\tau-s)L_{ss}}\mathbb{P}\Big(U^{lin}\cdot \nabla  U^{per}+U^{per}\cdot \nabla U^{lin}+U^{lin}\cdot\nabla U^{lin}+U^{per}\cdot\nabla U^{per} \Big)\dif s\Big\|_{X}
	\\&\lesssim e^{T(a+\varepsilon_0)}\|U^{per}\|_X^2+e^{Ta}\|U^{per}\|_X+e^{T(a-\eps_0)}.
\end{align*}
Hence, it is sufficient to estimate  the following terms
$$\int_{-\infty}^{\tau}e^{(\tau-s)L_{ss}}\mathbb{P}\Big(U\cdot \nabla  U^{per}+U^{per}\cdot \nabla U
+U\cdot \nabla U^{lin} +U^{lin}\cdot \nabla U \Big)\dif s.$$
Using \eqref{est} and Lemma \ref{lem:Lss} it holds for $0<\delta<\eps_0$
\begin{equation}\label{time int}
	\aligned
	&
	\Big\|\int_{-\infty}^{\tau}e^{(\tau-s)L_{ss}}\mathbb{P}(U\cdot \nabla  U^{per}+U^{per}\cdot \nabla U )\dif s\Big\|_{H^N}
	\\&\lesssim \int_{-\infty}^{\tau}\frac{ e^{(\tau-s)(a+\delta)+s(2a+\varepsilon_0+\eps)}}{(\tau-s)^{{1/2}}}\|U^{per}\|_X\dif s
	\lesssim e^{\tau(2a+\varepsilon_0+\eps)}\|U^{per}\|_X.
	\endaligned
\end{equation}
Using \eqref{r:Ulin},
 Lemma \ref{lem:Lss} and \eqref{est} we derive
\begin{align*}&
	\Big\|\int_{-\infty}^{\tau}e^{(\tau-s)L_{ss}}\mathbb{P}(U\cdot \nabla U^{lin} +U^{lin}\cdot \nabla U )\dif s\Big\|_{H^N}
	\\&\lesssim \int_{-\infty}^{\tau}\frac{e^{(\tau-s)(a+\delta)}e^{s(2a+\varepsilon)}}{(\tau-s)^{{1/2}}}\dif s
	\lesssim e^{\tau(2a+\varepsilon)}.
\end{align*}
Combining the above estimates we  choose $T$ very negative to apply fix point argument in a small ball in $X$ to construct a solution for \eqref{eqper} (see  \cite[Proposition 4.5]{ABC22} for more details).
\end{proof}
By \eqref{r:Ulin} and Proposition \ref{pro:2} we find that $U^{lin}+U^{per}\neq0$ as the convergence rate to $-\infty$ is different. Then $\tilde u(t,x)=\frac1{\sqrt{t}}\tilde U(\tau,\xi)$ gives the second Leray solution to \eqref{eq1} as the regularity of $\tilde u$ is the same as in \cite{ABC22}.

As noted above, the solutions $\tilde u$ and $u+\bar u$ can be extended to Leray--Hopf solutions on $[0,1]$. Hence, we deduce the following results.

\bt\label{th:1}
Let  $\bar f$ be the force obtained in \cite[Theorem 1.3]{ABC22} and let $f\in Y$ be arbitrary.
There exist two distinct suitable Leray--Hopf solutions $\tilde u$ and $u+\bar u$ to the Navier--Stokes equations on $[0,1]\times\mR^3$ with body force $\bar f+f$ and initial data $u_0\equiv0$.
\et

\bc\label{c:1}
For any $\eps>0$ there exist  $h$ with  $\|h\|_{L^1((0,1);L^2)}\leq \eps$,   and two distinct suitable Leray--Hopf solutions $\tilde u_1$ and $\tilde u_2$ to the Navier--Stokes equations on $[0,1]\times\mR^3 $ with body force $h$ and initial data $u_0\equiv0$.
\ec

\begin{proof}

For any $\eps>0$ there is  $f_\eps\in Y$ obtained  by convolution with a mollifier and a suitable cut-off near $t=0$ such that
\begin{align*}
	\|f_\eps+\bar f\|_{L^1(0,1;L^2)}\leq \eps.
\end{align*}
Choosing $h=f_\eps+\bar f$, the result follows  from Theorem \ref{th:1}.
\end{proof}

\bc \label{c:2}
For any $f\in L^1(0,1;L^2)$ and $\eps>0$ there exist $g\in L^1(0,1;L^2)$ with
$$\|g-f\|_{L^1(0,1;L^2)}\leq \eps,$$
and two distinct suitable Leray--Hopf solutions $\tilde u_1$ and $\tilde u_2$ to the Navier--Stokes equations on $[0,1]\times\mR^3$ with body force $g$ and initial data $u_0\equiv0$.
\ec

\begin{proof}
For any $\eps>0$ we could find $g_\eps\in Y$ by convolution with a mollifier and a suitable cut-off near $t=0$ such that
\begin{align*}
	\|g_\eps- f\|_{L^1(0,1;L^2)}\leq \eps/2.
\end{align*}
Choosing $g=f_{\eps/2}+\bar f+g_\eps$, the result is a consequence of Theorem \ref{th:1} and the fact that $f_{\eps/2}+g_\eps\in Y$, where $f_{\eps/2}$ was defined in the  proof of Corollary  \ref{c:1}.
\end{proof}

\section{General initial conditions}\label{s:6}

In this final section, we show a simple extension of the main result of \cite{ABC22} to general initial conditions in $L^{2}$ based on an approximate controllability argument from \cite{Fla97}. More precisely, we prove the following.

\begin{theorem}\label{thm:6}
Let $u_{0}\in L^{2}_{\sigma}$. There exists a body force $f=f_{u_0}$ so that the deterministic forced Navier--Stokes equations on $[0,1]\times\R^{3}$ admit two distinct Leray--Hopf solutions with initial condition $u_{0}$.
\end{theorem}

\begin{proof}
The idea is as follows. First, we show that there is a force $\tilde f$, a time $2T^{*}>0$ and  a Leray--Hopf solution $\tilde u$ to the deterministic forced Navier--Stokes equations with this force on the time interval $[0,2T^{*}]$ such that $\tilde u(2T^{*})=0$. Second, taking the final value as the initial condition on the next time interval,  we employ the technique of  \cite{ABC22} on  $[2T^{*},2T^{*}+T]$ to obtain two different Leray--Hopf solutions on $[0,2T^{*}+T]$ starting from the initial condition $u_{0}$.

In the first step, we start with an arbitrary Leray--Hopf solution $\tilde u $ to the Navier--Stokes equations with zero force and the initial condition  $u_{0}$. We can  modify the solution $\tilde u$, while keeping the same notation $\tilde u$,  so that there is a time  for which the solution belongs to $H^{2}$. Indeed, there is a time $T_{1}>0$ such that $\tilde u(T_{1})\in H^{1}$. We start $\tilde u(T_1)$ for a new unique local strong solutions
so that on some interval $[T_{1},T^{*}]$ the solution $\tilde u \in C ([T_{1},T^{*}]
; H^1) \cap L^2 (T_{1},T^{*} ; H^2)$.
In other words, making $T^{*}$ smaller if necessary, we may assume without loss
of generality that $\tilde{u} (T^{*}) \in H^2$.

Next, we intend to find a force $\tilde{f}$ so that $\tilde{u}$ extends to a
solution of the forced Navier--Stokes equations on $[0, 2T^{\ast}]$, so that
$\tilde{u} (2 T^{\ast}) = 0$.
We simply define by linear
interpolation
\[ \tilde{u} (t) : = \frac{2 T^{\ast} - t}{T^{\ast}} \tilde{u} (T^{\ast}),
   \qquad t \in [T^{*}, 2 T^{\ast}] . \]
Clearly, $\tilde{u} \in C ([T^{\ast}, 2 T^{\ast}] ; H^2)$ and since
\[ \| \mathbb{P} (\tilde{u} \cdummy \nabla \tilde{u}) \|_{L^2} \lesssim \|
   \tilde{u} \|_{L^{\infty}} \| \tilde{u} \|_{H^1} \lesssim \| \tilde{u}
   \|_{H^2} \| \tilde{u} \|_{H^1} \]
we deduce
\[ \tilde{f} \assign \partial_t \tilde{u} - \Delta \tilde{u} +\mathbb{P}
   (\tilde{u} \cdummy \nabla \tilde{u}) \in C ([T^{\ast}, 2 T^{\ast}] ; L^2) .
\]
Letting $\tilde{f} = 0$ on $[0, T^{\ast}]$, we therefore found a solution
$\tilde{u}$ to the forced Navier--Stokes equations with force $\tilde{f}$ on
$[0, 2 T^{\ast}]$ which satisfies the energy inequality on $[0,T_{1}]$, belongs to $C ([T_{1}, 2 T^{\ast}] ; H^1) \cap L^2 (T_{1}, 2
T^{\ast} ; H^2)$ and has the terminal value $\tilde{u} (2 T^{\ast}) = 0$. The
regularity of strong solution particularly implies that the energy inequality
holds true on the full time interval $[0,2T^{*}]$.

Finally, the construction from \cite{ABC22} permits to find a force and to extend the solution in a non-unique manner to some interval $[2
T^{\ast}, 2 T^{\ast} + T]$. Further extension to $[0,1]$ by usual Leray--Hopf solutions is immediate. The proof is complete.
\end{proof}

\end{document}